\documentclass[11pt,reqno]{amsart}\usepackage{amsmath,amscd,amstext,amsthm,amsfonts,latexsym}
\usepackage[dvips]{graphicx}
\usepackage{graphicx, graphics}

\usepackage{a4wide}
\usepackage{amsmath,amssymb,amsthm,mathrsfs}
\usepackage{color}
\usepackage{enumerate}
\usepackage{enumitem}
\theoremstyle{plain}
\newtheorem{theorem}{Theorem}[section]
\newtheorem{proposition}[theorem]{Proposition}
\newtheorem{lemma}[theorem]{Lemma}
\newtheorem{corollary}[theorem]{Corollary}

\theoremstyle{definition}
\newtheorem{remark}[theorem]{Remark}

\newtheorem{definition}[theorem]{Definition}

\newtheorem*{condition}{Condition}

\newcommand{\RR}{\field{R}}

\def\nn{\ensuremath{\mathscr N}}

\usepackage[T2A,T1]{fontenc}
\DeclareSymbolFont{cyrillic}{T2A}{cmr}{m}{n}
\DeclareMathSymbol{\D}{\mathalpha}{cyrillic}{196}

\newcommand{\N}{\mathbb{N}}

\newcommand{\R}{\mathbb{R}}

\def\I{\ensuremath{{\bf 1}}}

\newcommand{\dist}{\operatorname{{dist}}}

\def\R{\ensuremath{\mathbb R}}
\def\N{\ensuremath{\mathbb N}}

\def\I{\ensuremath{{\bf 1}}}
\def\e{{\ensuremath{\rm e}}}

\def\S{\ensuremath{\mathcal S}}
\def\RR{\ensuremath{\mathcal R}}

\def\B{\ensuremath{\mathcal B}}
\def\M{\ensuremath{\mathcal M}}

\def\l{{\rm Leb}}
\def\P{\ensuremath{\mathcal P}}
\def\p{\ensuremath{\mathbb P}}

\def\A{\ensuremath{A^{(q)}}}
\def\AA{\ensuremath{\mathcal A}}
\def\AAk{\ensuremath{\mathcal A_{q,n}^{(\kappa)}}}

\def\X{\mathcal{X}}
\def\ie{{\em i.e.}, }

\def\dist{\ensuremath{\text{dist}}}

\def\eps{\varepsilon}

\def\cyl{\text{Z}}

\def\cv{\ensuremath{\text {Cor}}}
\newcommand{\dif}{\mathrm{d}}

\author[D. Azevedo]{Davide Azevedo}
\address{Davide Azevedo\\Centro de Matem\'{a}tica da Universidade do Porto\\ Rua do
Campo Alegre 687\\ 4169-007 Porto\\ Portugal}
\email {davidemsa@gmail.com}

\author[A. C. M. Freitas]{Ana Cristina Moreira Freitas}
\address{Ana Cristina Moreira Freitas\\ Centro de Matem\'{a}tica \&
Faculdade de Economia da Universidade do Porto\\ Rua Dr. Roberto Frias \\
4200-464 Porto\\ Portugal} \email{amoreira@fep.up.pt}

\author[J. M. Freitas]{Jorge Milhazes Freitas}
\address{Jorge Milhazes Freitas\\ Centro de Matem\'{a}tica \& Faculdade de Ci\^encias da Universidade do Porto\\ Rua do
Campo Alegre 687\\ 4169-007 Porto\\ Portugal}
\email{jmfreita@fc.up.pt}
\urladdr{http://www.fc.up.pt/pessoas/jmfreita}

\author[F.B. Rodrigues]{Fagner Bernardini Rodrigues}
\address{Fagner Bernardini Rodrigues\\Instituto de Matem\'atica - Universidade Federal do Rio Grande do Sul\\
Av. Bento Gonçalves, 9500 - Prédio 43-111 - Agronomia\\
Caixa Postal 15080\
91509-900 Porto Alegre - RS - Brasil
} \email{fagnerbernardini@gmail.com }

\thanks{ACMF was partially supported by FCT grant SFRH/BPD/66174/2009. JMF was partially supported by FCT grant SFRH/BPD/66040/2009. Both these grants are financially supported by the program POPH/FSE. DA, ACMF, JMF are partially  supported by FCT (Portugal) project PTDC/MAT/120346/2010, which is funded by national and European structural funds through the programs  FEDER and COMPETE. FBR was supported by BREUDS, a Brazilian-European partnership of the FP7-PEOPLE-2012-IRSES program, with project number 318999, which is supported by an FP7 International International Research Staff Exchange Scheme (IRSES) grant of the European Union. All authors were partially supported by CMUP (UID/MAT/00144/2013), which is funded by FCT (Portugal) with national (MEC) and European structural funds through the programs FEDER, under the partnership agreement PT2020. All authors thank Mike Todd for careful reading and suggestions.}


\begin{document}

\title{Clustering of extreme events created by multiple correlated maxima}

\date{\today}


\maketitle
\begin{abstract}
We consider stochastic processes arising from dynamical systems by evaluating an observable function along the orbits of the system. The novelty is that we will consider observables achieving a global maximum value (possible infinite) at multiple points with special emphasis for the case where these maximal points are correlated or bound by belonging to the same orbit of a certain chosen point. These multiple correlated maxima can be seen as a new mechanism creating clustering. We recall that clustering was intimately connected with periodicity when the maximum was achieved at a single point. We will study this mechanism for creating clustering and will address the existence of limiting Extreme Value Laws, the repercussions on the value of the  Extremal Index, the impact on the limit of Rare Events Points Processes, the influence on clustering patterns and the competition of domains of attraction. We also consider briefly and for comparison purposes multiple uncorrelated maxima. The systems considered include expanding maps of the interval such as Rychlik maps but also maps with an indifferent fixed point such as Manneville-Pommeau maps.
\end{abstract}

\section{Introduction}

In the past few years, the study of Extreme Value Theory (EVT) for dynamical systems has been a subject of much interest. We mention in particular the inspiring paper of Collet \cite{C01} and refer to \cite{F13} for a review on further developments and references. Extreme events (which occur with small probability and for that reason are also called rare events) are characterised by abnormally high observations that are identified as exceedances of high thresholds.

In this dynamical setting the observation data comes from starting the system at a certain initial state and evaluate a certain observable function along the subsequent states through which the system goes while time goes by. Then, exceedances of a high threshold correspond to entrances or hits of the orbits of the system to some designated target sets on the phase space. As the threshold 
increases, the respective target sets shrink. This fact explains the connection between the existence of distributional limits, called Extreme Value Laws (EVL), for the partial maximum of observations when the thresholds increase to the possible maximum value, and the existence of distributional limits, called Hitting Times Statistics (HTS), for the waiting time before hitting the corresponding target sets as they shrink. This connection hinted in \cite{C01} was formally established in \cite{FFT10, FFT11}.

In the study of HTS, in most cases, the neighbourhoods are either cylinder sets or metric balls and shrink to a point $\zeta$ in the phase space. In the study of EVL the observable $\varphi$ has a global maximum at $\zeta$ and, in almost all cases, the exceedances correspond to metric balls around $\zeta$. The limiting laws for HTS and EVL were proved to be the same in \cite{FFT10, FFT11} and typically one obtains a standard exponential distribution $H(\tau)=1-\e^{-\tau}$, with $\tau\geq 0$. Typically, here, is used in the sense that in most results it is shown that for almost every point $\zeta$ (with respect to the invariant measure) one gets a standard exponential HTS and EVL. See for example \cite{HSV99, C01, BSTV03, HNT12, CC13,PS14} and \cite{S09} for an excellent review.

Following the work of Hirata \cite{H93} on Axiom A diffeomorphisms, it is known that at periodic points, \ie when $\zeta$ is a periodic point, a parameter $0\leq \theta\leq 1$ appears in the asymptotic distribution: $H(\tau)=1-\e^{-\theta\tau}$. This observation was further developed in the paper \cite{HV09}, where a compound Poisson distribution for the number of hits to target sets composed of unions of dynamical cylinders was obtained for $\psi$ mixing measures. Then in \cite{FFT12}, using the relation between HTS and EVL, a new technique based on suitably adjusted dependent conditions allowed to show the existence of EVL (and consequently HTS) in the case $\zeta$ is a periodic point and the target sets are metric balls for systems with a strong form of decay of correlations. Due to the connection with EVL the parameter $0\leq \theta\leq 1$ appearing in the limit  distribution $H(\tau)=1-\e^{-\theta\tau}$ was identified to be the Extremal Index (EI), a parameter already appearing in the context of classical EVT and which can be seen as the inverse of the average cluster size. 

Moreover, in \cite{FFT12}, for uniformly expanding systems such as the doubling map, a dichotomy was shown which states that either $\zeta$ is periodic and we have an EI with a very precise formula depending on the expansion rate at $\zeta$, or for every non-periodic $\zeta$, we have an EI equal to 1 (which means no clustering). The dichotomy was then obtained for more general systems such as: conformal repellers \cite{FP12}, systems with spectral gaps for the Perron-Frobenius operator \cite{K12} and systems with strong decay of correlations \cite{AFV14}.

The study of EVL and HTS can be enhanced by considering  Rare Events Point Processes (REPP). These point processes keep record of the exceedances of the high thresholds by counting the number of such exceedances on a rescaled time interval. At typical points it was known that the REPP converge in distribution to a standard Poisson process. From \cite{HV09}, at periodic points, we expect that the REPP converge to compound Poisson process instead. In \cite{FFT13}, a study regarding the convergence of REPP was performed for metric balls as target sets around periodic points $\zeta$ and for non-uniformly hyperbolic systems. The limit process obtained for the convergence in distribution of the REPP in \cite{FFT13} was a compound Poisson process which could be described as combination of Poisson process ruling the positions of the clusters in the time line being that these positions are marked by the cluster sizes ruled by a geometric multiplicity distribution. Both components are determined by the value of EI $\theta$.

Again, as proved in \cite{AFV14}, for continuous systems with strong decay of correlations a dichotomy regarding the convergence of REPP holds: either it converges to a compound Poisson distribution at repelling periodic points $\zeta$, with a geometric multiplicity distribution, or it converges to a standard Poisson process at every other non-periodic point $\zeta$. Moreover, as already seen in \cite{FFT13}, the clustering pattern observed at repelling periodic points $\zeta$ obeys a very rigid pattern:  it consists of `bulk' of strictly decreasing exceedances observed at precise fixed times, corresponding to the period.

Contrary to the customary case studies, in which the set of maximising points $\mathcal M$ of the observable $\varphi$ is reduced to a single point $\zeta$, in this paper we will consider the case of multiple maxima, with special emphasis for the case of correlated maxima, where $\mathcal M$ consists of finite number of points bound together by belonging to the same orbit. This binding gives a mechanism of creating clustering of exceedances, which we will study and explore below. We will also consider the case of uncorrelated maxima but the focus will be turned to the correlated case.

We remark that in the literature not all examples regarding the study of rare events are reduced to the cases in which the target sets shrink to a point. For example, in \cite{CCC09} the authors consider a sets generated by a proper subshift of finite type of a one-sided irreducible and aperiodic shift of finite type and in \cite{KL09} the authors consider target sets shrinking to the diagonal of the phase space of a system obtained by coupling two expanding interval maps. The setting which probably most resembles ours is that studied in \cite{HNT12}, where the authors, in the EVL approach, consider observables with multiple maxima that are chosen independently as typical points for the invariant probability measure. This way, for these uncorrelated maxima, they obtained the standard exponential distribution as limiting EVL, which means that there is no clustering of exceedances.

In classical Extreme Value Theory, it is usual to use normalising thresholds $(u_n)_{n\in\N}$ that are linear sequences depending on a parameter $y$, say $u_n=y/a_n+b_n$. As explained in \cite{FFT10}, the behaviour of the observable $\varphi$ as a function of the distance to $\zeta$ is responsible for determining the type of EVL that applies: Gumbel, Fr\'echet or Weibull. Hence, an interesting question that immediately arises when considering multiple maxima is the competition between the different types.

By taking multiple maxima over the same orbit we will see the following consequences:
\begin{enumerate}

\item  appearance of clustering not caused necessarily by periodic orbits;

\item the possibility of creating different clustering patterns (not reduced to a bulk of strictly decreasing observations over the threshold);

\item the possibility of affecting the EI when we already consider periodic points;

\item different multiplicity distributions  for the limit of REPP;

\item competition between different types of distributions.

\end{enumerate}

The systems considered include systems with decay of correlations against $L^1$ observables such as Rychlik maps \cite{R83} and piecewise uniformly expanding maps in higher dimensions studied in \cite{S00} but also non-uniformly expanding maps such as intermittent maps.

We believe that this study will be a precursor of further developments regarding more sophisticated maximal sets $\mathcal M$, that could be, for example, submanifolds of codimension 1. In fact, some work in that direction is already being done by some of the authors and we believe it carries a large potential of applications since it will more easily accommodate physical observables in applications like to meteorology, where the set  $\mathcal M$ plays the role of a critical set, which is the source of abnormal and possibly catastrophic events. Understanding the geometry and the recurrence properties of $\mathcal M$, as we will see here in a much simpler context, will provide knowledge about the extremal behaviour like for example the clustering patterns that could help devise early warning systems.

An example of possible application is to structural failures. The anticipated study of the type of the clustering patterns of a certain natural phenomenon is of crucial importance, on one side to predict the likelihood of such failure and on, another side, to help designing the material and structures to stand stronger against the natural causes they have to face.

The unfolding of such possibilities will give deeper understanding of the possible outputs of EVT for dynamical systems, as it will open a door for modelling physical phenomena with some underlying periodic effect, which is often difficult to perceive. Moreover, again, these short recurrence mechanisms have an enormous potential as a source of examples for the classical EVT of stochastic processes and serving as a model for several practical situations.

\section{The setting and background}
Take a system  $(\mathcal{X},\mathcal B,\mu,f)$, where $\mathcal{X}$ is a Riemannian manifold, $\mathcal{B}$ is the Borel
$\sigma$-algebra, $f:\mathcal{X}\to\mathcal{X}$ is a measurable map
and $\mu$ an $f$-invariant probability measure.
Suppose that the time series $X_0, X_1,\ldots$ arises from such a system simply by evaluating a given  observable $\varphi:\mathcal{X}\to\mathbb{R}\cup\{\pm\infty\}$ along the orbits of the system, or in other words, the time evolution given by successive iterations by $f$:
\begin{equation}
\label{eq:def-stat-stoch-proc-DS} X_n=\varphi\circ f^n,\quad \mbox{for
each } n\in {\mathbb N}.
\end{equation}
Clearly, $X_0, X_1,\ldots$ defined in this way is not necessarily an independent sequence.  However, $f$-invariance of $\mu$ guarantees that this stochastic process is stationary.
\vspace{0.5cm}\\
We suppose that the r.v. $\varphi:\mathcal{X}\to\mathbb{R}\cup\{\pm\infty\}$
has $N$ global maxima $\xi_1,\dots,\xi_N\in \mathcal{X}$ (we allow $\varphi(\xi_1)=+\infty$) and assume they all belong to the orbit of the point $\zeta$. (See \eqref{eq:maxima-bind} below).
Suppose $\varphi$ and $\mu$
are sufficiently regular in the following sense:

\begin{enumerate}
\item[{(R1)}]
for $u$ sufficiently close to $u_F:=\varphi(\xi_i)$ $(i\in \{1,\ldots,N\})$,
\begin{equation*}
\label{def:U}
U(u):=\{x\in\mathcal{X}:\; \varphi(x)>u\}=\{X_0>u\}
\end{equation*}
corresponds to a union of balls  centered at the points $\xi_i$, i.e., $U(u)=\bigcup_{i=1}^NB_{\varepsilon_i}(\xi_i)$
with $\varepsilon_i=\varepsilon_i(u)$. Moreover, the quantity $\mu(U(u))$, as a function of $u$, varies continuously on a neighborhood of $u_F$.

\end{enumerate}

\subsection{Extreme Value Laws}

In this paper, we will use an extreme value approach rather than an hitting times approach, which we have already mentioned to be two sides of the same coin as can be fully appreciated in \cite{FFT10,FFT11}.

We are interested in studying the extremal behaviour of the stochastic process $X_0, X_1,\ldots$ which is tied to the occurrence of exceedances of high levels $u$. The occurrence of an exceedance at time $j\in\mathbb{N}_0$ means that the event $\{X_j>u\}$ occurs, where $u$ is close to $u_F$. Observe that a realisation of the stochastic process $X_0, X_1,\ldots$ is achieved if we pick, at random and according to the measure $\mu$, a point $x\in\mathcal{X}$, compute its orbit and evaluate $\varphi$ along it. Then, saying that an exceedance occurs at time $j$ means that the orbit of the point $x$ hits one of the balls in $U(u)$ at time $j$, i.e., $f^j(x)\in B_{\varepsilon_i}(\xi_i)$ for some $i\in\{1,...,N\}$.

Given a stochastic process $X_0, X_1, \ldots$ we define a new sequence of random variables  $M_1, M_2,\ldots$ given by
\begin{equation}
\label{eq:Mn-definition}
M_n=\max\{X_0, X_1, \ldots, X_{n-1}\}.
\end{equation}

On the independent context, \ie when the stochastic process $X_0, X_1,\ldots$ is a sequence of independent and identically distributed (i.i.d.) r.v., the first statement regarding $M_n$ is that $M_n$ converges almost surely (a.s.) to $u_F$. Then, the next natural question is whether we can find a distributional limit for $M_n$, when conveniently normalised.
\begin{definition}
We say that we have an \emph{Extreme Value Law} (EVL) for $M_n$ if there is a non-degenerate d.f. $H:\R\to[0,1]$ with $H(0)=0$ and,  for every $\tau>0$, there exists a sequence of levels $u_n=u_n(\tau)$, $n=1,2,\ldots$,  such that
\begin{equation}
\label{eq:un}
  n\mu(X_0>u_n)\to \tau,\;\mbox{ as $n\to\infty$,}
\end{equation}
and for which the following holds:
\begin{equation}
\label{eq:EVL-law}
\mu(M_n\leq u_n)\to \bar H(\tau),\;\mbox{ as $n\to\infty$.}
\end{equation}
where the convergence is meant at the continuity points of $H(\tau)$.
\end{definition}

The motivation for using a normalising sequence $(u_n)_{n\in\N}$ satisfying \eqref{eq:un} comes from the case when $X_0, X_1,\ldots$ are independent and identically distributed (i.i.d.). In this setting, it is clear that $\mu(M_n\leq u)= (F(u))^n$, where $F$ is the d.f. of $X_0$. Hence, condition \eqref{eq:un} implies that
\begin{equation}
\label{eq:iid-maxima}
\mu(M_n\leq u_n)= (1-\mu(X_0>u_n))^n\sim\left(1-\frac\tau n\right)^n\to\e^{-\tau},
\end{equation}
as $n\to\infty$. Moreover, the reciprocal is also true (see \cite[Theorem~1.5.1]{LLR83} for more details). Note that in this case $H(\tau)=1-\e^{-\tau}$ is the standard exponential d.f.

\subsection{The existence of Extreme Value Laws}\label{section:EVL}

In \cite{FFT15}, the authors synthesised the conditions in \cite{FF08a}, in the absence of clustering, and in \cite{FFT12}, in the presence of clustering, to a couple of general conditions that apply to to general stationary stochastic processes, both in the presence and absence of clustering, which allow to prove the existence of EVL. Moreover, these conditions are particularly tailored to the application of dynamical systems and follow from a strong form of decay of correlations (against $L^1$ observables) to be defined below.

In what follows for every $A\in\mathcal B$, we denote the complement of $A$ as $A^c:=\mathcal X\setminus A$.

For some $u\in\R$, $q\in \N$, we define the events:
\begin{align}
\label{eq:U-A-def}
U(u)&:= \{X_0>u\},\nonumber\\
\AA_q(u)&:=U(u)\cap\bigcap_{i=1}^{q}f^{-i}(U(u)^c)=\{X_0>u, X_1\leq u, \ldots, X_q\leq u\}.
\end{align}
where $\AA_q(u)$ corresponds to the case where we have an extreme event at time zero that is not followed by another one up to time $t=q$. This is a condition clearly pointing to the absence of clustering. We also set $\AA_0(u):=U(u)$, $U_n:=U(u_n)$ and $\AA_{q,n}:=\AA_q(u_n)$, for all $n\in\N$ and $q\in\N_0$.
Let
\begin{equation}
\label{def:thetan}
\theta_n:=\frac{\mu\left(\AA_{q,n}\right)}{\mu(U_n)}.
\end{equation}

Let $B\in\B$ be an event. For some $s\geq0$ and $\ell\geq 0$, we define:
\begin{equation}
\label{eq:W-def}
\mathscr W_{s,\ell}(B)=\bigcap_{i=\lfloor s\rfloor}^{\lfloor s\rfloor+\max\{\lfloor\ell\rfloor-1,\ 0\}} f^{-i}(B^c).
\end{equation}
The notation $f^{-i}$ is used for the preimage by $f^i$. We will write $\mathscr W_{s,\ell}^c(B):=(\mathscr W_{s,\ell}(B))^c$.
Whenever is clear or unimportant which event $B\in\B$ applies, we will drop the $B$ and write just $\mathscr W_{s,\ell}$ or $\mathscr W_{s,\ell}^c$.
Observe that
\begin{equation}
\label{eq:EVL-HTS}
\mathscr W_{0,n}(U(u))=\{M_n\leq u\}\qquad \mbox{and}\qquad f^{-1}(\mathscr W_{0,n}(B))=\{r_{B}>n\}.
\end{equation}

\begin{condition}[$\D_q(u_n)$]\label{cond:D} We say that $\D(u_n)$ holds for the sequence $X_0,X_1,\ldots$ if for every  $\ell,t,n\in\N$,
\begin{equation}\label{eq:D1}
\left|\mu\left(\AA_{q,n}\cap
 \mathscr W_{t,\ell}\left(\AA_{q,n}\right) \right)-\mu\left(\AA_{q,n}\right)
  \mu\left(\mathscr W_{0,\ell}\left(\AA_{q,n}\right)\right)\right|\leq \gamma(q,n,t),
\end{equation}
where $\gamma(q,n,t)$ is decreasing in $t$ for each $n$ and, there exists a sequence $(t_n)_{n\in\N}$ such that $t_n=o(n)$ and
$n\gamma(q,n,t_n)\to0$ when $n\rightarrow\infty$.
\end{condition}

For some fixed $q\in\N_0$, consider the sequence $(t_n)_{n\in\N}$, given by condition  $\D(u_n)$ and let $(k_n)_{n\in\N}$ be another sequence of integers such that
\begin{equation}
\label{eq:kn-sequence}
k_n\to\infty\quad \mbox{and}\quad  k_n t_n = o(n).
\end{equation}

\begin{condition}[$\D'_q(u_n)$]\label{cond:D'q} We say that $\D'_q(u_n)$
holds for the sequence $X_0,X_1,X_2,\ldots$ if there exists a sequence $(k_n)_{n\in\N}$ satisfying \eqref{eq:kn-sequence} and such that
\begin{equation}
\label{eq:D'rho-un}
\lim_{n\rightarrow\infty}\,n\sum_{j=q+1}^{\lfloor n/k_n\rfloor-1}\mu\left( \AA_{q,n}\cap f^{-j}\left(\AA_{q,n}\right)
\right)=0.
\end{equation}
\end{condition}

From \cite[Corollary~2.4]{FFT15} follows that if the stochastic process satisfies both conditions $\D_q(u_n)$ and $\D'_q(u_n)$ for some $q\in\N_0$, where $(u_n)_{n\in\N}$ satisfies \eqref{eq:un}, then $\lim_{n\to\infty}\mu(M_n\leq u_n)=\e^{-\theta\tau}$, whenever the limit $\theta=\lim_{n\to\infty}\theta_n$ exists.

In this approach, it is rather important to observe the prominent role played by condition $\D'_q(u_n)$. In particular, note that if condition $\D'_q(u_n)$ holds for some particular $q=q_0\in\N_0$, then condition $\D'_q(u_n)$ holds for all $q\geq q_0$, which also implies that if the limit of $\theta_n$ in \eqref{def:thetan} exists for such $q=q_0$ it will also exist for all $q\geq q_0$ and takes always the same value. This suggests that in trying to find the existence of EVL, one should try the values $q=q_0$ until we find the smallest one that makes $\D'_q(u_n)$ hold. With that purpose, as in \cite[Theorem~3.7]{FFT15}, we consider the following.
Let $A\in\B$. We define a function that we refer to as \emph{first hitting time function} to $A$, denoted by $r_A:\X\to\N\cup\{+\infty\}$ where
\begin{equation}
\label{eq:hitting-time}
r_A(x)=\min\left\{j\in\N\cup\{+\infty\}:\; f^j(x)\in A\right\}.
\end{equation}
The restriction of $r_A$ to $A$ is called the \emph{first return time function} to $A$. We define the \emph{first return time} to $A$, which we denote by $R(A)$, as the infimum of the return time function to $A$, \ie
\begin{equation}
\label{eq:first-return}
R(A)=\inf_{x\in A} r_A(x).
\end{equation}
Assume that there exists $q\in\N_0$ such that
\begin{equation}
\label{eq:q-def EVL}
q:=\min\left\{j\in\N_0: \lim_{n\to\infty}R(A_n^{(j)})=\infty\right\}.
\end{equation}
Then such $q$ is the natural candidate to try to show the validity of $\D'_q(u_n)$.


\subsection{Rare Events Point Processes}

A more sophisticated way of studying rare events consists in studying Rare Events Point Processes (REPP). These point processes keep record of the exceedances of the high thresholds $u_n$ by counting the number of such exceedances on a rescaled time interval. For every $A\subset\R$ we define
\[
\nn_{u_n}(A):=\sum_{i\in A\cap\N_0}\I_{X_i>u_n}.
\]

In order to provide a proper framework of the problem we introduce next the necessary formalism to state the results regarding the convergence of point processes. We recommend the book of Kallenberg \cite{K86} for further reading.

Consider the interval $[0,\infty)$ and its Borel $\sigma$-algebra $\B_{[0,\infty)}$. Let $x_1, x_2, \ldots \in [0,\infty)$ and define
$$
\nu=\sum_{i=1}^\infty \delta_{x_i},
$$
where $\delta_{x_i}$ is the Dirac measure at $x_i$, \ie for every $A\in\B_{[0,\infty)}$, we have that $\delta_{x_i}(A)=1$ if $x_i\in A$ or
$\delta_{x_i}(A)=0$, otherwise. The measure $\nu$ is said to be a counting measure on $[0,\infty)$. Let $\M_p([0,\infty))$ be the space of counting measures on $([0,\infty), \B_{[0,\infty)})$.  We equip this space with the vague topology, i.e., $\nu_n\to \nu$ in $\M_p([0,\infty))$ whenever $\nu_n(\psi)\to \nu(\psi)$ for any continuous function $\psi:[0,\infty)\to \R$ with compact support.  A \emph{point process} $N$ on $[0,\infty)$ is a random element of $\M_p([0,\infty))$, \ie let $(X, \mathcal \B_X, \mu)$ be a probability space, then any measurable map $N:X\to \M_p([0,\infty))$ is a point process on $[0,\infty)$.

To give a concrete example of a point process, which in particular will appear as the limit of the REPP, we consider:
\begin{definition}
\label{def:compound-poisson-process}
Let $T_1, T_2,\ldots$ be  an i.i.d. sequence of r.v. with common exponential distribution of mean $1/\theta$. Let  $D_1, D_2, \ldots$ be another i.i.d. sequence of r.v., independent of the previous one, and with d.f. $\pi$. Given these sequences, for $J\in\B_{[0,\infty)}$, set
$$
N(J)=\int \I_J\;d\left(\sum_{i=1}^\infty D_i \delta_{T_1+\ldots+T_i}\right),
$$
where $\delta_t$ denotes the Dirac measure at $t>0$. Let $X$ denote the space of all possible realisations of $T_1, T_2,\ldots$ and   $D_1, D_2, \ldots$, equipped with a product $\sigma$-algebra and measure, then $N:X\to \M_p([0,\infty))$ is a point process which we call a compound Poisson process of intensity $\theta$ and multiplicity d.f. $\pi$.
\end{definition}
\begin{remark}
\label{rem:poisson-process}
Throughout the paper, the multiplicity will always be integer valued which means that $\pi$ is completely defined by the values $\pi_k=\p(D_1=k)$, for every $k\in\N_0$. Note that, if $\pi_1=1$ and $\theta=1$, then $N$ is the standard Poisson process and, for every $t>0$, the random variable $N([0,t))$ has a Poisson distribution of mean $t$.
\end{remark}

In order to define the REPP we need to rescale time. 
We do it using the factor $v_n:=1/\p(X_0>u_n)$ given by Kac's Theorem. However, before we give the definition, we need some more formalism. Let $\S$ denote the semi-ring of subsets of  $\R_0^+$ whose elements
are intervals of the type $[a,b)$, for $a,b\in\R_0^+$. Let $\RR$
denote the ring generated by $\S$. Recall that for every $J\in\RR$
there are $k\in\N$ and $k$ intervals $I_1,\ldots,I_k\in\S$ such that
$J=\cup_{i=1}^k I_j$. In order to fix notation, let
$a_j,b_j\in\R_0^+$ be such that $I_j=[a_j,b_j)\in\S$. For
$I=[a,b)\in\S$ and $\alpha\in \R$, we denote $\alpha I:=[\alpha
a,\alpha b)$ and $I+\alpha:=[a+\alpha,b+\alpha)$. Similarly, for
$J\in\RR$ define $\alpha J:=\alpha I_1\cup\cdots\cup \alpha I_k$ and
$J+\alpha:=(I_1+\alpha)\cup\cdots\cup (I_k+\alpha)$.

\begin{definition}
We define the \emph{rare event point process} (REPP) by
counting the number of exceedances during the (rescaled) time period $v_nJ\in\RR$, where $J\in\RR$. To be more precise, for every $J\in\RR$, set
\begin{equation}
\label{eq:def-REPP} N_n(J):=\nn_{u_n}(v_nJ)=\sum_{j\in v_nJ\cap\N_0}\I_{X_j>u_n}.
\end{equation}
\end{definition}

We will see that the REPP considered here converge to a compound Poisson process. For completeness, we define here what we mean by convergence of point processes (see \cite{K86} for more details).
\begin{definition}
\label{def:convergence-point-processes}
Let $(N_n)_{n\in\N}:X\to  \M_p([0,\infty))$ be a sequence of point processes defined on a probability space $(X,\mathcal B_X, \mu)$ and let $N:Y \to  \M_p([0,\infty))$ be another point process defined on a possibly distinct probability space $(Y,\mathcal B_Y, \nu)$. We say that $N_n$ converges in distribution to $N$  if $\mu\circ N_n^{-1}$ converges weakly to $\nu\circ N^{-1}$, \ie for every continuous function $\varphi$ defined on $\M_p([0,\infty))$, we have $\lim_{n\to\infty}\int \varphi d \mu\circ N_n^{-1}=\int \varphi d \mu\circ N^{-1}$.  We write $N_n \stackrel{\mu}{\Longrightarrow} N $.
\end{definition}

\begin{remark}
\label{rem:convergence-point-processes}
It can be shown that $(N_n)_{n\in\N}$ converges in distribution to $N$ if the sequence of vector r.v.  $(N_n(J_1), \ldots, N_n(J_k))$ converges in distribution to $(N(J_1), \ldots, N(J_k))$, for every $k\in\N$ and all $J_1,\ldots, J_k\in\S$ such that $N(\partial J_i)=0$ a.s., for $i=1,\ldots,k$.
\end{remark}

Note that
\begin{equation}
\label{eq:rel-HTS-EVL-pp}
\{\nn_{u_n}([0,n))=0\}=\{M_n\leq u_n\},
\end{equation}
hence the limit distribution of $M_n$ can be easily recovered from the convergence of the REPP.

\subsection{The convergence of  REPP}
In  \cite{FFR15}, the authors are preparing a synthesised version of the conditions in \cite{FFT10}, in the absence of clustering, and in \cite{FFT13}, in the presence of clustering, to a couple of general conditions that apply to general stationary stochastic processes, both in the presence and absence of clustering, which allow to prove the convergence of REPP. Moreover, these conditions, as the previous $\D_q(u_n)$ and $\D'(u_n)$, are particularly tailored to the application of dynamical systems and follow from a strong form of decay of correlations (against $L^1$ observables) to be defined below.

Before we state the next condition we introduce some notation which will be useful in the rest of the paper.
Let $\left(U^{(\kappa)}(u)\right)_{\kappa\geq0}$ be a sequence  of nested open sets defined by:
\begin{align*}
\label{def:U-k}
&U^{(0)}(u):=U(u)
\\
&\AA_q^{(0)}(u):=\{X_0>u,X_1\leq u,...,X_q\leq u\}=U^{(0)}(u)\cap\bigcap_{i=1}^{q}f^{-i}\left((U^{(0)}(u))^c\right).
\\&U^{(\kappa)}(u):=U^{(\kappa-1)}(u)-A_q^{(\kappa-1)}(u)\quad \mbox{and}\quad \AA^{(\kappa)}_q(u):=U^{(\kappa)}(u)\cap\bigcap_{i=1}^{q}f^{-i}\left((U^{(\kappa)}(u))^c\right).
\end{align*}
Let
\begin{align*}
\theta_n:=\frac{\mu(\AA_{q,n}^{(0)})}{\mu(U_n)} \mbox{ and } \pi_n(\kappa)=\frac{\mu(\AA^{(\kappa)}_{q,n})-
\mu(\AA^{(\kappa+1)}_{q,n})}{\mu(\AA^{(0)}_{q,n})}.
\end{align*}
We notice that the set $\AA^{(0)}_{q,n}$ was defined in Section~\ref{section:EVL} and named $\AA_{q,n}$.
Since we need more
\begin{condition}[$\D_q(u_n)^*$]\label{cond:Dp*}We say that $\D_q(u_n)^*$ holds
for the sequence $X_0,X_1,X_2,\ldots$ if for any integers $t, \kappa_1,\ldots,\kappa_\varsigma$, $n$ and
 any $J=\cup_{i=2}^\varsigma I_j\in \mathcal R$ with $\inf\{x:x\in J\}\ge t$,
 \[ \left|\mu\left(\AA_{q,n}^{(\kappa_1)}\cap \left(\cap_{j=2}^\varsigma \nn_{u_n}(I_j)=\kappa_j \right) \right)-\mu\left(\AA_{q,n}^{(\kappa_1)}\right)
  \mu\left(\cap_{j=2}^\varsigma \nn_{u_n}(I_j)=\kappa_j \right)\right|\leq \gamma(n,t),
\]
where for each $n$ we have that $\gamma(n,t)$ is nonincreasing in $t$  and
$n\gamma(n,t_n)\to0$  as $n\rightarrow\infty$, for some sequence
$t_n=o(n)$.
\end{condition}
In  \cite[Corollary~1]{FFR15} the authors show that under conditions $\D_q(u_n)^*$  and $\D_q'(u_n)$,
where $(u_n)_n$ satisfies \eqref{eq:un},
the REPP $(N_n)_{n\in\mathbb N}$ converges in distribution to a point process $N$ whenever the limits $\theta=\lim_{n\to\infty}\theta_n$
and $\pi(\kappa)=\lim_{n\to\infty}\pi_n(\kappa)$ exist (for all $\kappa\in\mathbb N$).

If one assumes that the maximum of $f$ occurs in a repelling periodic point $\zeta$ and  the  probability  $\mu$
is sufficiently regular then
\begin{enumerate}
\item[(R2)]
the fact that $\zeta$ is repelling means that we have backward contraction implying that there exists $0<\theta<1$ so that $\bigcap_{j=0}^i f^{-jp}(X_0>u)$ is another ball of smaller radius around $\zeta$ with $$\mu\left(\bigcap_{j=0}^i f^{-jp}(X_0>u)\right)\sim(1-\theta)^i\mu(X_0>u),$$ for all $u$ sufficiently large.
\end{enumerate}
Condition $(R2)$ guarantees that $\theta_n=\frac{\mu\big(\AA_{q,n}^{(0)}\big)}{\mu\big(U_n\big)}$ converges
to $\theta$. Moreover, by \cite[Theorem~1]{FFT13} the REPP $(N_n)_{n\in\mathbb N}$ converges in distribution to a compound Poisson process
$N$ with intensity  $\theta$ and multiplicity distribution given by $\pi(\kappa)=\theta(1-\theta)^{\kappa-1}$.

\section{Multiple correlated maxima}

In this paper the novelty of the approach resides inn the fact that instead of considering observables $\varphi:\X\to\R$ achieving a global maximum at a single point $\zeta\in\X$, we assume that the maximum is achieved at $N$ points denoted by $\xi_1, \xi_2\ldots,\xi_N$, where in almost all cases (except in Sections~\ref{subsec:non-correlated-non-periodic} and \ref{subsec:non-correlated-periodic}), all these points are correlated by belonging to the orbit of a certain point $\zeta\in\X$, \ie there exist $m_1< m_2<\ldots< m_N\in\N_0$ such that
\begin{equation}
\label{eq:maxima-bind}
\xi_i=f^{m_i}(\zeta),
\end{equation}
 after possible relabelling of the maximal points. Let $\mathcal M:=\{\xi_1, \xi_2\ldots,\xi_N\}$ be the set of maximal points of $\varphi$. We will see that the binding between the maximal points expressed in \eqref{eq:maxima-bind} is responsible for a fake periodic behaviour which creates clustering of exceedances. Moreover, this mechanism of creating clustering turns out to be more flexible and allows to obtain easily different multiplicity distributions for the limit of the REPP (we recall that at repelling periodic points one obtains a Geometric multiplicity distribution) and different clustering patterns (we recall that at repelling periodic points the clusters consist of a strictly decreasing bulk of exceedances occurring at precise periodic intervals).

\subsection{Further assumptions on the observable}
Besides having $N$ global maxima $\xi_1,\dots,\xi_N$, we assume that the observable $\varphi$ also
satisfies
\begin{equation}\label{def:observable}
\varphi(x)=h_i(\dist (x,\xi_i)),\; \forall x\in B_{\varepsilon_i}(\xi_i),\;\;\; i\in\{1,...,N\}
\end{equation}
where $B_{\varepsilon_i}(\xi_i) \cap B_{\varepsilon_j}(\xi_j)=\emptyset$ and  the function
$h_i:[0,+\infty)\to \mathbb{R}\cup\{\infty\}$ is such that 0 is a global maximum ($h_i(0)$ may be $+\infty$),
 $h_i$ is a strictly decreasing bijection $h_i : V \to W$ in a neighbourhood $V$ of 0; and has one of the following three types of behaviour:

\begin{enumerate}
  \item  Type 1: there exists some strictly positive function $g : W \to \mathbb{R}$ such that for all $y \in\mathbb{R} $
\begin{equation}
\label{eq:type1}
\lim_{ s\to g_1(0)}\frac{h_i^{-1}(s+yg(s))}{h_i^{-1}(s)}= e^{-ˆ'y};
\end{equation}
 \item  Type 2: $h_i(0) = +\infty$ and there exists $\beta>0$ such that for all $y > 0$
 \begin{equation}
\label{eq:type2}
\lim_{s\to\infty}\frac{h_i^{-1}(sy)}{h_i^{-1}(s)}=y^{-\beta};
 \end{equation}
  \item Type 3: $h_i(0) = D < +\infty$ and there exists $\gamma > 0$ such that for all $y > 0$
  \begin{equation}
  \label{eq:type3}
  \lim_{s\to0}\frac{ h^{-1}_i(Dˆ-sy)}{h^{-1}_i(D-ˆ's)}= y^\gamma.
  \end{equation}
\end{enumerate}

\begin{remark}
\label{rem:R1}
Note that as long as the invariant measure has no atoms, then under the assumptions above on the observable we have that condition ($R1$) is easily satisfied.
\end{remark}

\subsection{Assumptions on the system and examples of application}

We assume that the system admits a first return time induced map with decay of correlations against $L^1$ observables. In order to clarify what is meant by the latter we define:
\begin{definition}[Decay of correlations]
\label{def:dc}
Let \( \mathcal C_{1}, \mathcal C_{2} \) denote Banach spaces of real valued measurable functions defined on \( \X \).
We denote the \emph{correlation} of non-zero functions $\phi\in \mathcal C_{1}$ and  \( \psi\in \mathcal C_{2} \) w.r.t.\ a measure $\p$ as
\[
\cv_\p(\phi,\psi,n):=\frac{1}{\|\phi\|_{\mathcal C_{1}}\|\psi\|_{\mathcal C_{2}}}
\left|\int \phi\, (\psi\circ f^n)\, \dif\p-\int  \phi\, \dif\p\int
\psi\, \dif\p\right|.
\]

We say that we have \emph{decay
of correlations}, w.r.t.\ the measure $\p$, for observables in $\mathcal C_1$ \emph{against}
observables in $\mathcal C_2$ if, for every $\phi\in\mathcal C_1$ and every
$\psi\in\mathcal C_2$ we have
 $$\cv_\p(\phi,\psi,n)\to 0,\quad\text{ as $n\to\infty$.}$$
  \end{definition}

We say that we have \emph{decay of correlations against $L^1$
observables} whenever  this holds for $\mathcal C_2=L^1(\p)$  and
$\|\psi\|_{\mathcal C_{2}}=\|\psi\|_1=\int |\psi|\,\dif\p$.

If a system already has decay of correlations against $L^1$ observables, then by taking the whole set $\X$ as the base for the first return time induced map, which coincides with the original system, then the assumption we impose on the system is trivially satisfied. Examples of systems with such property include:
\begin{itemize}

\item Uniformly expanding maps on the circle/interval (see \cite{BG97});

\item Markov maps (see \cite{BG97});

\item Piecewise expanding maps of the interval with countably many branches like Rychlik maps (see \cite{R83});

\item Higher dimensional piecewise expanding maps studied by Saussol in \cite{S00}.

\end{itemize}

\begin{remark}
\label{rem:C-space}
In the first three examples above the Banach space $\mathcal C_1$ for the decay of correlations can be taken as the space of functions of bounded variation. In the fourth example  the Banach space $\mathcal C_1$ is the space of functions with finite quasi-H\"older norm studied in \cite{S00}. We refer the readers to \cite{BG97,S00} or \cite{AFV14} for precise definitions but mention that if $J\subset \R$ is an interval then $\I_J$ is of bounded variation and its BV-norm is equal to 2, \ie $\|\I_J\|_{BV}=2$ and if $A$ denotes a ball or an annulus then $\I_A$ has a finite quasi-H\"older norm.
\end{remark}

Although the examples above are all in some sense uniformly hyperbolic, we can consider non-uniformly hyperbolic systems, such as intermittent maps, which admit a `nice' first return time induced map over some subset $Y\subset \X$, called the base of the induced map. To be more precise, consider the usual original system as $f:\X\to \X$ with an ergodic $f$-invariant probability measure $\mu$, choose a subset $Y\subset \X$ and consider $F_Y:Y\to Y$ to be the first return map $f^{r_Y}$ to $Y$ (note that $F$ may be undefined at a zero Lebesgue measure set of points which do not return to $Y$, but most of these points are not important, so we will abuse notation here).  Let $\mu_Y(\cdot)=\frac{\mu(\cdot \cap Y)}{\mu(Y)}$ be the conditional measure on $Y$.  By Kac's Theorem $\mu_Y$ is $F_Y$-invariant.

In \cite{BSTV03}, the authors show that for typical points $\zeta\in Y$ the HTS to balls around $\zeta$ is the same both for the original and induced maps. This means, in particular, that the existence of EVL for the induced map implies the existence of an EVL for the original system, at typical points, and the limit is actually the same. In \cite{FFT13}, a similar statement was obtained but this time for the convergence of the REPP at periodic points having the same compound Poisson limit. In  \cite{HayWinZwe14} the authors show that regardless of the point $\zeta\in Y$ taken, the existence of a limit for the HTS (or EVL) for the induced first return time map implies the existence of a limit for the HTS (or EVL) of the original system, at $\zeta$, and the limits coincide. This was generalised for the convergence of REPP in \cite{FFTV15}. Namely,
setting $v_n^Y=1/\mu_Y(X>u_n)$, for the induced process $X_i^Y$,
$$N_n^Y(J):=\nn_{u_n}^Y(v_n^YJ)=\sum_{j\in v_n^YJ\cap\N_0}\I_{X_j^Y>u_n}.$$
 Denote the speeded up return time $r_A$ by $r_{A, Y}$ and the induced measure on $Y$ by $\mu_Y$.
\begin{theorem}[{\cite[Theorem~3]{FFTV15}}]
For $\eta>0$, setting $J_\eta:=\cup_{s\in J}B_{\eta}(s)$, we assume that $N(J_\eta)$ is continuous in $\eta$, for all small $\eta$.
$$N_n^Y\stackrel{\mu_Y}{\Longrightarrow} N \text{ as } n\to\infty \text{ implies } N_n \stackrel{\mu}{\Longrightarrow} N \text{ as } n\to\infty.$$
\label{thm:pp_ret_orig}
\end{theorem}

This means that if we have a system that admits a first return time induced map with decay of correlations against $L^1$, then as long as the points $\xi_1,\ldots, \xi_N\in Y$, the convergence of the REPP (which implies the existence of EVL or HTS) for the original systems is determined by the convergence of the corresponding REPP for the induced map. Hence, we are reduced to proving the convergence of the REPP for systems with decay of correlations against $L^1$ observables.
This fact motivates the following:
\vspace{0.5cm}

\noindent \textbf{Assumption A}
\emph{Let $f:\X\to\X$ be a systems with summable decay of correlations against $L^1$ observables, \ie for all $\varphi\in\mathcal C_1$ and $\psi\in L^1$, then $\cv(\varphi,\psi,n)\leq \rho_n$, with $\sum_{n\geq\N}\rho_n<\infty$. Moreover, we assume that there exists $C>0$ such that for all $n\in\N$, we have $\I_{\A_n}, \I_{\AAk}\in\mathcal C_1$ and $\|\I_{\A_n}\|_{\mathcal C_1}, \|\I_{\AAk}\|_{\mathcal C_1}\leq C$. }
\vspace{0.5cm}

Among the examples of systems with these `nice' induced maps we mention  the \emph{Manneville-Pomeau} (MP) map equipped with an absolutely continuous invariant probability measure.  The form for such maps given in \cite{LSV99, BSTV03} is,   for $\alpha\in (0,1)$,
\begin{equation*}
f=f_\alpha(x)=\begin{cases} x(1+2^\alpha x^\alpha) & \text{ for } x\in [0, 1/2)\\
2x-1 & \text{ for } x\in [1/2, 1]\end{cases}
\end{equation*}
Members of this family of maps are often referred to as Liverani-Saussol-Vaienti maps since their actual equation was first introduced in \cite{LSV99}. Let $\P$ be the \emph{renewal partition}, that is the partition defined inductively by $\cyl\in \P$ if $\cyl=[1/2, 1)$ or $f(\cyl)\in \P$.  Now let $Y\in \P$ and let $F_Y$ be the first return map to $Y$ and $\mu_Y$ be the conditional measure on $Y$.  It is well-known that $(Y, F_Y, \mu_Y)$ is a Bernoulli map and hence, in particular, a Rychlik system (see \cite{R83} or \cite[Section~3.2.1]{AFV14} for the essential information about such systems).

\subsection{Existence of distributional limits under Assumption A}
\label{subsec:existence}

As explained in \cite[Section~5.1]{F13}, conditions $\D_q(u_n)$ and $\D_q(u_n)^*$ are designed to follow easily from decay of correlations. In fact, if we choose $\phi=\I_{\AA_{q,n}}$ and $\psi=\I_{\mathscr W_{0,\ell}\left(\AA_{q,n}\right)}$, in the case of $\D_q(u_n)$, and  $\phi=\I_{\AAk}$ and $\psi=\I_{\cap_{j=2}^\varsigma \nn_{u_n}(I_j)=\kappa_j }$, in the case of $\D_q(u_n)^*$, we have that we can take $\gamma(n,t)=C\rho_t$. Hence, by taking $(t_n)_n$ such that $\lim_{n\to\infty}n\rho_{t_n}=0$, we get that conditions $\D_q(u_n)$ and $\D_q(u_n)^*$ are trivially satisfied.

Now, we turn to condition $\D'_q(u_n)$.
Taking $\phi=\psi=\I_{\AA_{q,n}}$ and since $\|\I_{\AA_{q,n}}\|_{\mathcal C_1}\leq C$ we easily get
\begin{align}
\label{eq:estimate1}
\mu\left(U_n\cap f^{-j}(U_n)\right) \le
 \left(\mu(\AA_{q,n})\right)^2+ \left\| \I_{\AA_{q,n}}\right\|_{\mathcal C_1} \left\| \I_{\AA_{q,n}}\right\|_{L^1(\mu)} \rho_j\leq  \left(\mu(\AA_{q,n})\right)^2+C\mu(\AA_{q,n})\rho_j.
\end{align}
Let $R_n:=R(\AA_{q,n})$. Using estimate \eqref{eq:estimate1} and  $n\mu(U_n)\to \tau$ as $n \to \infty$ it follows that there exists some constant $D>0$ such that
\begin{align*}
n\sum_{j=q+1}^{\lfloor n/k_n \rfloor}& \mu(\AA_{q,n}\cap f^{-j}(\AA_{q,n})) = n\sum_{j=R_n}^{\lfloor n/k_n \rfloor} \mu(\AA_{q,n}\cap f^{-j}(\AA_{q,n}))\\
&\le n\big\lfloor\tfrac {n}{k_n}\big\rfloor\mu(\AA_{q,n})^2 +n\,C\mu(\AA_{q,n}) \sum_{j=R_n}^{\lfloor n/k_n \rfloor}\rho_j\le \frac{\left(n\mu(\AA_{q,n})\right)^2}{k_n} +n\,C\mu(\AA_{q,n}) \sum_{j=R_n}^{\infty}\rho_j\\
& \leq D \left(\frac{\tau^2}{k_n}+\tau \sum_{j=R_n}^{\infty}\rho_j \right)\xrightarrow[n\to\infty]{}0,
\end{align*}
if we check that $\lim_{n\to\infty}R_n=+\infty$.

\section{Maxima along a non-periodic orbit}
\label{sec:nonperiodic}

In this section we consider the case where $\zeta$ is a
non-periodic point. Our goal is to see what are the effects on the REPP
when we compare it with the process obtained by an observable which attains its maximum at
a unique point. As we will see, despite $\zeta$ being non-periodic, the existence of multiple maxima along its orbit affects, for example, the extremal index and the multiplicity distribution.

Note that by ($R1$), we know that, for $n$ sufficiently large, $U_n$ is the union of $N$ disjoint balls around each $\xi_i$ that we denote by $U_n(\xi_i)=U_n^{(0)}(\xi_i)$, for $i=1,\ldots,N$, so that
\begin{equation}
\label{eq:connected-components-U}
U_n=\bigcup_{i=1}^N U_n(\xi_i).
\end{equation}
We begin by choosing a value of $q$ for which condition $\D'(u_n)$ can be checked. From the computations in Section~\ref{subsec:existence} we need to verify that $R(\AA_{q,n})\to\infty$, as $n\to\infty$ for such choice of $q$.

\begin{lemma}
\label{lem:continuity}
Assume that $f$ satisfies Assumption A and that $\zeta $ is a non-periodic point and that $f$ is continuous in every point
of the orbit of $\zeta$, namely $\zeta,f(\zeta),f^2(\zeta)...$. Let $\xi_1,\ldots,\xi_N$ be as in \eqref{eq:maxima-bind}. Let $q=m_N-m_1$, then $\lim_{n\to\infty}R(\AA_{q,n})=\infty$.
\end{lemma}
\begin{proof}
 Let $q=m_N-m_1$. By definition of $\AA_{q,n}$ we have that $f^j(\AA_{q,n})\cap \AA_{q,n}\subset f^j(\AA_{q,n})\cap U_n=\emptyset$, for all $j=1, \ldots, q$. Hence, we only need to check that the same holds true for all $j=q+1, \ldots, q+L$, where $L$ is a given arbitrarily large integer. For that we will use the continuity of $f$ over the orbit of $\zeta$.
Let
$$
\epsilon:=\displaystyle \min_{1\leq i\leq L,\ 1\leq j\leq N} \dist(f^i(\xi_N),\xi_j).
$$
Using the continuity over the orbit of $\zeta$ and in particular of $\xi_N$, for every $1\leq i\leq L$, consider the balls $B_{\delta_i}(\xi_N)$ such that $f^i(B_{\delta_i}(\xi_N))\subset B_{\eps/2}(f^i(\xi_N))$ and consider the open set $D_N=\cap_{i=1}^L B_{\delta_i}(\xi_N)$. Now, for each $j=1,\ldots, N-1$ let $D_j$ be a ball around $\xi_j$ so that $f^{m_N-m_j}(D_j)\subset D_N$.

Finally, recalling ($R1$) take $n$ sufficiently large so that each $U_n(\xi_i)\subset D_i\cap B_{\eps/2}(\xi_i)$ so that $U_n\subset \cup_{i=1}^N D_i\cap B_{\eps/2}(\xi_i)$. Recall that $\AA_{q,n}\subset U_n$ and then by construction for every $j=1,\ldots, L$ we have that $f^{q+j}(\AA_{q,n})\subset \cup_{i=1}^N B_{\eps/2}(f^{q+j}(\xi_i))$ and since by definition of $\eps$, we have $B_{\eps/2}(f^{q+j}(\xi_i))\cap U_n=\emptyset$, then the result follows.
\end{proof}

\subsection{The clustering effect and estimates for the extremal index}
When the
point $\zeta$ is an absolute maximum and $\zeta $ is non-periodic we know that
the extremal index $\theta$  is equal to 1. The situation is completely different
when the maximum occurs at more than one point along the orbit of $\zeta$.
\begin{proposition}\label{prop:extremal-index-nonperiod}
Assume that $f$ satisfies Assumption A. Let $X_0,X_1,...$ be given by \eqref{eq:def-stat-stoch-proc-DS}, where the observable
$\varphi$ has $N$ maximal points $\xi_1,\ldots,\xi_N$ related by \eqref{eq:maxima-bind} and satisfies \eqref{def:observable}, with $\zeta\in \mathcal X$ a non-periodic
point. In this case
the extremal index $\theta$ is given by
$$
\theta=\lim_{n\to\infty} \frac{\sum_{i=1}^N \mu\left(U_{n}^{(0)}(\xi_i)-\bigcup_{j=i+1}^N f^{-(m_j-m_i)}(U_{n}^{(0)}(\xi_j))\right)}
{\sum_{i=1}^N \mu\left(U_{n}^{(0)}(\xi_i)\right)},
$$
whenever the limit exists.
\end{proposition}
\begin{proof}
In this case we have that $q=m_N-m_1$ and $\AA_{q,n}^{(0)}=\bigcup_{i=1}^N \AA_{q,n}^{(0)}(\xi_i)$, where
$$
\AA_{q,n}^{(0)}(\xi_i)=U_{n}^{(0)}(\xi_i)-\bigcup_{j=i+1}^N f^{-(m_j-m_i)}(U_{n}^{(0)}(\xi_j)).
$$
By the definition of
$\displaystyle\theta=\lim_{n\to\infty} \frac{\mu(\AA_{q,n}^{(0)})}{\mu(U_n^{(0)})}$, the conclusion follows.
\end{proof}

In order to illustrate the richness of scenarios that the formula for the EI of the previous proposition encompasses we consider below some particular examples. By condition ($R1$), we know that, for $n$ sufficiently large, $U_n=U_n^{(0)}$ is the union of $N$ disjoint balls around each $\xi_i$ that we denote by $U^{(0)}_n(\xi_i)$, for $i=1,\ldots,N$, so that $U_n^{(0)}=\cup_{i=1}^N U^{(0)}_n(\xi_i)$. Depending on the different shapes of $h_i$ in \eqref{def:observable}, the balls $U^{(0)}_n(\xi_i)$ may have different sizes and shapes. Moreover, if $f$ is continuous along the orbit of $\zeta$, then for $1\leq i<j\leq N$ we have $f^{m_j-m_i}(U^{(0)}_n(\xi_i))\cap U^{(0)}_n(\xi_j)\neq\emptyset$ but, depending on the expansion rate and geometric distortion of $f^{m_j-m_i}$, it is not a priori clear if the one of the sets contains the other or not. Hence, in order to simplify, we assume that for $1\leq i<j\leq N$

\begin{enumerate}\label{eq:back-forward-contraction}
\item[{($R3$)}]$
U_{n}^{(0)}(\xi_i) \subset f^{-(m_j-m_i)}\left(U_{n}^{(0)}(\xi_j)\right)
 \mbox{ or } f^{-(m_j-m_i)}\left(U_{n}^{(0)}(\xi_j)\right) \cap \partial U_{n}^{(0)}(\xi_i)=\emptyset,$
\end{enumerate}
where the notation $\partial B$ is used to denote the boundary of the set $B$.

\begin{remark}
\label{rem:boundary}
Note that for all $1\leq i<j\leq N$, we have $f^{-(m_j-m_i)}\left(U_{n}^{(0)}(\xi_j)\right) \cap U_{n}^{(0)}(\xi_i)\neq \emptyset$. Hence, $f^{-(m_j-m_i)}\left(U_{n}^{(0)}(\xi_j)\right) \cap \partial U_{n}^{(0)}(\xi_i)=\emptyset$ means that the connected component of $f^{-(m_j-m_i)}\left(U_{n}^{(0)}(\xi_j)\right)$ that intersects $U_{n}^{(0)}(\xi_i)$ is completely contained in $U_{n}^{(0)}(\xi_i)$.
\end{remark}

In the next result we treat the case where the condition ($R3$) is valid. Before
we assert the corollary we introduce some notation.

Let $I_1$ be the set of indices $i\in  \{1,\dots N-1\}$ such that
\begin{equation}\label{eq:I1} \bigcup_{\ell=i+1}^{N} f^{-(m_l-m_i)}(U_{n}^{(0)}(\xi_\ell)) \cap\partial U_{n}^{(0)}(\xi_i)=\emptyset.\end{equation}
Under ($R3$), for every $i\in I_1$, there exists $j_{i}\in \{i+1,\dots N\}$ such that
\begin{equation*}
\bigcup_{\ell=i+1}^{N} f^{-(m_l-m_i)}(U_{n}^{(0)}(\xi_l))\cap U_{n}^{(0)}(\xi_i)=f^{-(m_{j_{i}}-m_i)}(U_{n}^{(0)}(\xi_{j_{i}}))\cap U_{n}^{(0)}(\xi_i).
\end{equation*}

\begin{corollary}\label{particular1}
If we are under the conditions of Proposition~\ref{prop:extremal-index-nonperiod} and the
condition $(R3)$ is satisfied,
 then the extremal index $\theta$ is given by
$$
\theta=\lim_{n\to\infty} \frac{\sum_{i\in I_1} \left[\mu\left(U_{n}^{(0)}(\xi_i)\right)- \mu\left(f^{-(m_{j_i}-m_i)}(U_{n}^{(0)}(\xi_{j_i}))\cap U_{n}^{(0)}(\xi_i)\right)\right]+\mu\left(U_{n}^{(0)}(\xi_N)\right)}
{\sum_{i=1}^N \mu\left(U_{n}^{(0)}(\xi_i)\right)},
$$
whenever the limit exists.
\end{corollary}

Let us now assume we are in the particular case where $\mu$ is absolutely continuous with respect to the Lebesgue measure and  its Radon-Nikodym density is sufficiently regular so that for all $x\in\X$ we have
\begin{equation}
\label{eq:density}
\lim_{\eps\to 0}\frac{\mu(B_\eps(x))}{\l(B_\eps(x))}=\frac{d\mu}{d\l}(x).
\end{equation}
Note that if $f$ is one dimensional smooth map and $\log(Df)$ is H\"older continuous then, as seen in \cite[Section~7.2]{FFT15}, formula \eqref{eq:density} holds.

\begin{corollary}\label{particular1cor} If we are under the hypothesis of Corollary \ref{particular1}  and $\mu$ is absolutely continuous and \eqref{eq:density} holds then
the extremal index $\theta$ is given by
$$
\theta=\lim_{n\to\infty} \frac{\displaystyle\sum_{i\in I_1} \left[\mu\left(U_{n}^{(0)}(\xi_i)\right)- \displaystyle\frac{1}{\left|\det Df^{m_{j_i}-m_i}(\xi_{i})\right|}\mu\left(U_{n}^{(0)}(\xi_{j_i})\right)\right]
+\mu\left(U_{n}^{(0)}(\xi_N)\right)}
{\displaystyle\sum_{i=1}^N \mu\left(U_{n}^{(0)}(\xi_i)\right)}.
$$
\end{corollary}
\begin{proof}
We only need to notice that, by the mean value theorem and \eqref{eq:density}
$$
\mu\left(f^{-(m_{j_i}-m_i)}(U_{n}^{(0)}(\xi_{j_i}))\cap U_{n}^{(0)}(\xi_i)\right)
\sim \frac{1}{\left|\det Df^{m_{j_i}-m_i}(\xi_{i})\right|}
\mu\left(U_{n}^{(0)}(\xi_{j_i})\right).
$$
\end{proof}

\subsection{Estimating the multiplicity distribution}

By Proposition 2.10 of \cite{FFR15} we have that the multiplicity
distribution of the limiting compound Poisson process for the REPP is given by the formula
\begin{equation}\label{def:mult-distr}
\pi(k)=\lim_{n\to\infty}\frac{\mu\left(\AA_{q,n}^{(k-1)}\right)-\mu\left(\AA_{q,n}^{(k)}\right)}{\mu\left(\AA_{q,n}^{(0)}\right)},
\end{equation}
whenever the limit exists.

Assume that ($R3$) holds.

Let $I_1$ be defined as in \eqref{eq:I1}. For every $i\in \{1,\dots,N\}$, we define $j_{i,0}=i$. Then, for $i\in \{1,\dots,N-k\}$ we inductively define $I_k$ in the following way.

Assuming that $I_k$ is defined as well as $j_{1,k-1}, \ldots, j_{N-k+2,k-1}$, then for every $i=1, \ldots, N-k+1$, we define 
$j_{i,k}$ to the index such that
\begin{equation}
\bigcup_{\ell=i+1}^{N} f^{-(m_{j_{\ell,k-1}}-m_i)}(U_{n}^{(0)}(\xi_{j_{\ell,k-1}}))\cap U_{n}^{(0)}(\xi_i)=f^{-(m_{j_{i,k}}-m_i)}(U_{n}^{(0)}(\xi_{j_{i,k}}))\cap U_{n}^{(0)}(\xi_i),
\end{equation}
when $i\in I_{k}$ and $j_{i,k}=i$, otherwise.

Now, we define $I_{k+1}$ as the set of indices $i\in  \{1,\dots N-k\}$ such that
$$ \bigcup_{\ell=i+1}^{N} f^{-(m_{j_{\ell,k}}-m_i)}(U_{n}^{(0)}(\xi_{j_{\ell,k}}))\cap \partial U_{n}^{(0)}(\xi_i)=\emptyset.$$

\begin{proposition}
\label{prop:multip-R3}
Assume that $f$ satisfies Assumption A. Let $X_0,X_1,...$ be given by \eqref{eq:def-stat-stoch-proc-DS}, where the observable
$\varphi$ has $N$ maximal points $\xi_1,\ldots,\xi_N$ related by \eqref{eq:maxima-bind} and satisfies \eqref{def:observable}, with $\zeta\in \mathcal X$ a non-periodic
point. Assume that conditin ($R3$) holds. The multiplicity distribution of the point process is given by equation \eqref{def:mult-distr} (whenever the limit exists),
where, for $k<N$,
\begin{align*}
&\mu(\AA^{(k)}_{q,n})=\mu\left(f^{-(m_{j_{N-k,k}}-m_{N-k})}(U_{n}^{(0)}(\xi_{j_{N-k,k}}))\cap U_{n}^{(0)}(\xi_{N-k})\right)
\\&+
\sum_{i\in I_{k+1}} \left[\mu\left(f^{-(m_{j_{i,k}}-m_i)}(U_{n}^{(0)}(\xi_{j_{i,k}})) \cap U_{n}^{(0)}(\xi_i)\right)-\mu\left(f^{-(m_{j_{i,k+1}}-m_i)}(U_{n}^{(0)}(\xi_{j_{i,k+1}}))\cap U_{n}^{(0)}(\xi_i)\right)\right]
\end{align*}
 and, for $k\geq N$, $\mu(\AA^{(k)}_{q,n})=0$.

\end{proposition}
\begin{proof}
We will prove by induction that, for $i\in I_{k+1}$,
$$\AA_{q,n}^{(k)}(\xi_i)= U_{n}^{(0)}(\xi_i)\cap f^{-(m_{j_{i,k}}-m_i)}(U_{n}^{(0)}(\xi_{j_{i,k}}))-f^{-(m_{j_{i,k+1}}-m_i)}(U_{n}^{(0)}(\xi_{j_{i,k+1}}))$$
and $$U_{n}^{(k+1)}(\xi_i)=f^{-(m_{j_{i,k+1}}-m_i)}(U_{n}^{(0)}(\xi_{j_{i,k+1}}))\cap U_{n}^{(0)}(\xi_i),$$
and, for $i\not\in I_{k+1}$, $\mu\left(\AA_{q,n}^{(k)}(\xi_i)\right)=0$ since $U_{n}^{(k)}(\xi_i)=U_{n}^{(k-1)}(\xi_i)$.

For $i\in I_1$ since $j_{\ell,0}=\ell$, we get
\begin{align*}
\AA_{q,n}^{(0)}(\xi_i)&=U_{n}^{(0)}(\xi_i)-\bigcup_{\ell=i+1}^{N} f^{-(m_\ell-m_i)}(U_{n}^{(0)}(\xi_\ell))
\\&=U_{n}^{(0)}(\xi_i)-\bigcup_{\ell=i+1}^{N} f^{-(m_{j_{\ell.0}}-m_i)}(U_{n}^{(0)}(\xi_{j_{\ell,0}}))
\\&=U_{n}^{(0)}(\xi_i)\cap f^{-(m_{j_{i,0}}-m_i)}(U_{n}^{(0)}(\xi_{j_{i,0}}))-f^{-(m_{j_{i,1}}-m_i)}(U_{n}^{(0)}(\xi_{j_{i,1}})),
\end{align*}
by definition of $j_{i,1}$, and, moreover, we have $U_{n}^{(1)}(\xi_i)=f^{-(m_{j_{i,1}}-m_i)}(U_{n}^{(0)}(\xi_{j_{i,1}}))\cap U_{n}^{(0)}(\xi_i)$.

Note that, for $i=N$, $\displaystyle\bigcup_{\ell=i+1}^{N} f^{-(m_\ell-m_i)}(U_{n}^{(0)}(\xi_l))=\emptyset$ and so
$$\AA_{q,n}^{(0)}(\xi_N)=U_{n}^{(0)}(\xi_N)=f^{-(m_{j_{N,0}}-m_i)}(U_{n}^{(0)}(\xi_{j_{N,0}})).$$

For $i\not\in I_1$, $i\not =N$, $U_{n}^{(0)}(\xi_i)-\bigcup_{\ell=i+1}^{N+i} f^{-(m_\ell-m_i)}(U_{n}^{(0)}(\xi_l))=\emptyset$, so $\mu\left(\AA_{q,n}^{(0)}(\xi_i)\right)=0$ and $\mu\left(U_{n}^{(1)}(\xi_i)\right)=\mu\left(U_{n}^{(0)}(\xi_i)\right)$.

By induction, for $i\in I_{k+1}$, we obtain
\begin{align*}
&\AA_{q,n}^{(k)}(\xi_i)=U_{n}^{(k)}(\xi_i)-\bigcup_{\ell=i+1}^{N} f^{-(m_\ell-m_i)}(U_{n}^{(k)}(\xi_l))
\\&=U_{n}^{(k)}(\xi_i)-\bigcup_{\ell=i+1}^{N} f^{-(m_\ell-m_i)}\left(f^{-(m_{j_{\ell,k}}-m_l)}(U_{n}^{(0)}(\xi_{j_{\ell,k}}))\cap U_{n}^{(0)}(\xi_i)\right),\,\mbox{by induction}
\\&=f^{-(m_{j_{i,k}}-m_i)}(U_{n}^{(0)}(\xi_{j_{i,k}})) \cap U_{n}^{(0)}(\xi_i)-\bigcup_{\ell=i+1}^{N} f^{-(m_{j_{\ell,k}}-m_i)}(U_{n}^{(0)}(\xi_{j_{\ell,k}})), \, \mbox{by induction}
\\&=f^{-(m_{j_{i,k}}-m_i)}(U_{n}^{(0)}(\xi_{j_{i,k}}))\cap U_{n}^{(0)}(\xi_i)-f^{-(m_{j_{i,k+1}}-m_i)}(U_{n}^{(0)}(\xi_{j_{i,k+1}})), \mbox{ by definition of $j_{i,k+1}$}.
\end{align*}
and in particular we also have  $U_{n}^{(k+1)}(\xi_i)=f^{-(m_{j_{i,k+1}}-m_i)}(U_{n}^{(0)}(\xi_{j_{i,k+1}}))\cap U_{n}^{(0)}(\xi_i)$.

For $i=N-k+1$, $\displaystyle\bigcup_{\ell=i+1}^{N} f^{-(m_\ell-m_i)}(U_{n}^{(k)}(\xi_l))=\emptyset$ and so
$$\AA_{q,n}^{(k)}(\xi_{N-k})=f^{-(m_{j_{N-k+1,k}}-m_{N-k})}(U_{n}^{(0)}(\xi_{j_{N-k+1,k}}))\cap U_{n}^{(0)}(\xi_{N-k}).$$

For $i\not\in I_{k+1}$, $i\not =N-k+1$ $\mu\left(\AA_{q,n}^{(k)}(\xi_i)\right)=0$ and
$$U_{n}^{(k+1)}(\xi_i)=f^{-(m_{j_{i,k}}-m_i)}(U_{n}^{(0)}(\xi_{j_{i,k}}))\cap U_{n}^{(0)}(\xi_i).$$
For $k>N+1$, $\mu\left(\AA_{q,n}^{(k)}(\xi_i)\right)=0$.

This concludes the proof.
\end{proof}

\begin{corollary}\label{particular2}Under the assumptions of Proposition~\ref{prop:multip-R3}. Assume that $\mu$ is absolutely continuous with respect to the Lebesgue measure and \eqref{eq:density} holds. Then the multiplicity distribution of the point process is given by formula \eqref{def:mult-distr},
where, for $k<N$,
\begin{align*}
\mu(\AA^{(k)}_{q,n})=\displaystyle\sum_{i\in I_{k+1}} \Big(\alpha_n(m_{j_{i,k}},m_i)-\alpha_n(m_{j_{i,k+1}},m_i)\Big)+\alpha_n(m_{j_{N-k,k}},m_{N-k})
\end{align*}
and
$$\alpha_n(a,b)=\displaystyle\frac{1}{|\det Df^{a-b}(f^b(\zeta))|}\mu \left(U_{n}^{(0)}(f^a(\zeta))\right).$$
\end{corollary}
\begin{proof}
Since $\mu$ is absolutely continuous with respect to the Lebesgue measure and \eqref{eq:density} holds, then
$$\mu\left(f^{-(m_{j_{i,k}}-m_i)}(U_{n}^{(0)}(\xi_{j_{i,k}}))\cap U_{n}^{(0)}(\xi_i)\right)=\displaystyle\frac{1}{\left|\det Df^{m_{j_{i,k}}-m_i}(\xi_{i})\right|}\mu \left(U_{n}^{(0)}(\xi_{j_{i,k}})\right)=\alpha_n(m_{j_{i,k}},m_i).$$
Replacing this in the previous proposition we obtain the conclusion
\end{proof}

\begin{corollary}
Assume that $\mu$ is absolutely continuous with respect to the Lebesgue measure and \eqref{eq:density} holds and
$$ f^{-(m_{i+1}-m_i)}\left(U_{n}^{(0)}(\xi_{i+1})\right) \subset U_{n}^{(0)}(\xi_i).$$
Then the multiplicity distribution of the point process is given by formula \eqref{def:mult-distr}
where, for $k<N$,
\begin{align*}
\mu(\AA^{(k)}_{q,n})=\sum_{i\in I_{k+1}} \Big(\alpha_n(m_{i+k},m_i)-\alpha_n(m_{i+k+1},m_i)\Big)+\alpha_n(m_N,m_{N-k})
\end{align*}
and $\alpha_n(a,b)$ is as in Corollary~\ref{particular2}.
\end{corollary}
\begin{proof}
Using the notation of Proposition~\ref{prop:multip-R3}, note that in this case $j_{i,k}=i+k$.
\end{proof}

\subsection{An example}

Let $(f,\mathbb{S}^1,Leb,\mathcal{B})$ be the dynamical system where
 $f:\mathbb{S}^1\to \mathbb{S}^1$ is the map given by $f(x)=2x \mod 1$. Take
$\zeta= \frac{\sqrt{2}}{16}$ and an observable $\varphi: \mathbb{S}^1\to \mathbb R\cup \{\infty\}$
such that
$$
\varphi(x)=\left\{
             \begin{array}{lll}
               -\log|x-\zeta|, & \hbox{ if }x\in B_\epsilon(\zeta), \\
               |x-f(\zeta)|^{\frac{-1}{2}}, & \hbox{ if }x\in B_\epsilon(f(\zeta)) \\
               |x-f^3(\zeta)|^{\frac{-1}{2}}, & \hbox{ if }x\in B_\epsilon(f^3(\zeta))\\
               0, & \hbox{ otherwie}.
             \end{array}
           \right.
$$
where we choose $\epsilon>0$, such that $ B_\epsilon(f^i(\zeta))\cap B_\epsilon(f^j(\zeta))=\emptyset$
for $ i,j\in\{0,1,3\} $ and $i\not=j$. In this case
we have that $u_F=\sup_{x\in  \mathbb{S}^1} \varphi(x)=\infty$ and it occurs at $\zeta, f(\zeta)$ and $f^3(\zeta)$.

Given $\tau>0$ and a sequence $(u_n)_{n\in\mathbb N}$ satisfying \eqref{eq:un} we have that
$$
U_n=(\zeta-e^{-u_n},\zeta+e^{-u_n})\cup(f(\zeta)-u_n^{-2},f(\zeta)+u_n^{-2})
\cup(f^3(\zeta)-u_n^{-2},f^3(\zeta)+u_n^{-2}).
$$
As before we let $U_n(\zeta)=(\zeta-e^{-u_n},\zeta+e^{-u_n})$, $U_n(f(\zeta))=(f(\zeta)-u_n^{-2},f(\zeta)+u_n^{-2})$ and $U_n(f^3(\zeta))=(f^3(\zeta)-u_n^{-2},f^3(\zeta)+u_n^{-2})$. We notice that for $n$ sufficiently large we have that
$$
(\zeta-e^{-u_n},\zeta+e^{-u_n})\varsubsetneqq f^{-1}\left((f(\zeta)-u_n^{-2},f(\zeta)+u_n^{-2})\right),
$$
and
$$
f^{-2}\left((f^3(\zeta)-u_n^{-2},f^3(\zeta)+u_n^{-2})\right)\cap U_n(f(\zeta))\varsubsetneqq \left(f(\zeta)-u_n^{-2},f(\zeta)+u_n^{-2}\right).
$$
Noting that $m_1=0,m_2=1, m_3=3$, then $I_1=\{2\}$ and $j_2=3$. So, by Corollary \ref{particular2} we have that

\begin{align*}
\theta&=\lim_{n\to\infty} \frac{\mu\left(U_{n}^{(0)}(f^{m_2}(\zeta))\right)- \mu\left(f^{-(m_3-m_2)}(U_{n}^{(0)}(f^{m_3}(\zeta)))\right)+\mu\left(U_{n}^{(0)}(f^{m_3}(\zeta))\right)}
{\sum_{i=1}^3 \mu\left(U_{n}^{(0)}(\xi_i)\right)}
\\&=\lim_{n\to\infty}\frac{u_n^{-2}-\displaystyle\frac{1}{Df^{2}(f^3(\zeta))}u_n^{-2}+u_n^{-2}}{e^{-u_n}+u_n^{-2}+u_n^{-2}}
=\frac{2-\frac{1}{4}}{2}=\frac{7}{8}.
\end{align*}

To compute the multiplicity distribution we need to determine the sets $I_k$, for all $k\in \mathbb{N}$. It is easy to see that, in this case,
$I_1=\{2\}$ and $I_k=\emptyset$, for all $k\geq 2$. We now need to determine the value of $j_{i,k}$, for all $k\in \mathbb{N}$ and all $i\in I_k$. By definition $j_{i,0}=i$, for all $i\in\mathbb{N}$. In addition, we can see that $j_{2,1}=3$. Since $m_1=0,m_2=1, m_3=3$, by Corollary \ref{particular2}, we have
\begin{align*}
\mu(\AA^{(0)}_{q,n})&=(\alpha(m_{j_{2,0},2})-\alpha(m_{j_{2,1},2})+\alpha_n(m_{j_{3,0}},m_3))2u_n^{-2}
\\&=(\alpha(m_2,m_2)-\alpha(m_3,m_2)+\alpha_n(m_3,m_3))2u_n^{-2}
\\&=2\left(u_n^{-2} -\left(\frac{1}{2}\right)^2u_n^{-2}+u_n^{-2}\right)=\frac{14}{4}u_n^{-2},
\\ \mu(\AA^{(1)}_{q,n})&=\alpha(m_{j_{2,1}},m_2)=\alpha(m_3,m_2)=\left(\frac{1}{2}\right)^22u_n^{-2},
\\ \mu(\AA^{(2)}_{q,n})&=\alpha(m_{j_{1,2}},m_1)=\alpha(m_1,m_1)=2e^{-u_n},
\\ \mu(\AA^{(k)}_{q,n})&=0 \text{ for }k\geq 3.
\end{align*}
Therefore,
\begin{align*}
\pi(1)&=\lim_{n\to\infty} \frac{\mu(\AA^{(0)}_{q,n})-\mu(\AA^{(1)}_{q,n})}{\mu(\AA^{(0)}_{q,n})}
=\lim_{n\to\infty}\frac{\displaystyle u_n^{-2}-\left(\frac{1}{2}\right)^2u_n^{-2}+u_n^{-2}-\left(\frac{1}{2}\right)^2u_n^{-2}}
{\displaystyle u_n^{-2} -\left(\frac{1}{2}\right)^2u_n^{-2}+u_n^{-2}}=\frac{6}{7},
\\\pi(2)&=\lim_{n\to\infty} \frac{\mu(\AA^{(1)}_{q,n})-\mu(\AA^{(2)}_{q,n})}{\mu(\AA^{(0)}_{q,n})}
=\lim_{n\to\infty}\frac{\displaystyle \left(\frac{1}{2}\right)^2u_n^{-2}-e^{-u_n}}
{\displaystyle u_n^{-2}-\left(\frac{1}{2}\right)^2u_n^{-2}+u_n^{-2}}=\frac{1}{7},
\\ \pi(3)&=\lim_{n\to\infty} \frac{\mu(\AA^{(2)}_{q,n})-\mu(\AA^{(3)}_{q,n})}{\mu(\AA^{(0)}_{q,n})}
=\lim_{n\to\infty}\frac{\displaystyle e^{-u_n}}
{\displaystyle u_n^{-2}-\left(\frac{1}{2}\right)^2u_n^{-2}+u_n^{-2}}=0,
\\ \pi(k)&=0 \text{ for } k\geq 4.
\end{align*}

\subsection{Non correlated maxima along non periodic orbits}
\label{subsec:non-correlated-non-periodic}

Although the main purpose of the paper is to study the effect of multiple maxima when they are bound by belonging to the same orbit, we note that if we suppose that $\varphi$ achieves a global maximum value at the points $\xi_1, \ldots, \xi_n$ but these maximal points have no intersecting orbits then one can also obtain a limiting EVL, which in the case of the maximal points being non-periodic, has no clustering and hence the EI equals 1.

Let $\mathcal O(x)$ denote the sequence of points that form the orbit of the point $x$, \ie $\mathcal O(x)=\{x, f(x), f^2(x), \ldots\}$ .

\begin{proposition}
Assume that $f$ satisfies Assumption A. Let $X_0,X_1,...$ be given by \eqref{eq:def-stat-stoch-proc-DS}, where the observable
$\varphi$ has $N$ maximal points $\xi_1,\ldots,\xi_N$. Assume that all the maximal points are non-periodic, that $f$ is continuous on $\cup_{i=1}^N\mathcal O(\xi_i)$ and that
\begin{equation}
\label{eq:non-intersection}
\bigcap_{i=1}^N \mathcal O(\xi_i)=\emptyset.
\end{equation}
Then $X_0,X_1,...$ satisfies conditions $\D_0(u_n)$ and $\D'_0(u_n)$ and consequently for $(u_n)_{n\in\N}$ satisfying \eqref{eq:un} then $\lim_{n\to\infty}\mu(M_n\leq u_n)=\e^{-\tau}$.
\end{proposition}

\begin{proof}
In this case we take $q=0$, meaning that $\AA_{q,n}=U_n$. Conditions $\D_0(u_n)$ and $\D'_0(u_n)$ follow from the computations in Section~\ref{subsec:existence} as long as one check that $\lim_{n\to\infty}R(U_n)=+\infty$. This last statement follows from a continuity argument very similar to the proof of Lemma~\ref{lem:continuity}.
\end{proof}

\begin{remark}
\label{rem:non-correlated}
This is in agreement with the result obtained in \cite{HNT12} for multiple uncorrelated maxima chosen as independent typical points of the invariant measure of non-uniformly hyperbolic systems admitting a Young tower with summable tails. Here, if the system is continuous and satisfies Assumption A (which is strictly contained in the class of systems considered in \cite{HNT12}), the statement can be reinforced since it applies to all non-periodic points chosen for maxima as long as they have non-intersecting orbits.
\end{remark}

\section{Multiple maxima lying on periodic orbit}
\label{sec:periodic}

In this section we consider the case where $\zeta$ is a
repelling periodic point of prime period $p$, meaning that there exists $p\in\N$ such that $f^p(\zeta)=\zeta$, being that $p$ is the smallest integer for which that happens, and the derivative $Df^p(\zeta)$ is defined and all its eigenvalues are strictly larger than 1. Note that in this case, if for example $\mu$ is absolutely continuous with respect to Lebesgue measure and \eqref{eq:density} holds then condition ($R_2$) is satisfied. From \cite{FFT12} we know that if the global maximum of $\varphi$ is attained only at $\zeta$ then we expect the existence of an EI equals to
$$
\theta=1-\frac1{|\det Df^p(\zeta)|},
$$
in the case $\mu$ is absolutely continuous with respect to Lebesgue and \eqref{eq:density} holds.
Moreover, by \cite{FFT13}, we know that the REPP converges to a compound Poisson process with a geometric multiplicity distribution given by
$$
\pi(\kappa)=\theta (1-\theta)^{\kappa-1},
$$
for all $\kappa\geq 1$.

In this section, our goal is to study the effect of having multiple maxima along the orbit  of the periodic point $\zeta$. We will see
the EI is affected (it decreases) and asymptotics of the \emph{REPP} is also affected.

We assume without loss of generality that $\varphi$ achieves a maximum at the points $\xi_1, \ldots, \xi_N$, where each $\xi_i$ is given by equation \eqref{eq:maxima-bind}, where $0=m_1<m_2<\ldots<m_N\leq p-1$, so that, in particular $\xi_1=\zeta$.

We begin by choosing a value of $q$ for which condition $\D'(u_n)$ can be checked. One easily anticipates that a suitable choice is taking $q=p$. From the computations in Section~\ref{subsec:existence} we need to verify that $R(\AA_{q,n})\to\infty$, as $n\to\infty$ for such choice of $q$.

\begin{lemma}
Assume that $\zeta $ is a repelling periodic point of prime period $p$. Let $\xi_1,\ldots,\xi_N$ be as in \eqref{eq:maxima-bind}, where $0=m_1<m_2<\ldots<m_N\leq p-1$. Let $q=p$, then $\lim_{n\to\infty}R(\AA_{q,n})=\infty$.

\end{lemma}

\begin{proof}
Let $g:=f^p$, and so every $\xi_i $ is a fixed point of $g$.
As $\zeta$ is a repelling periodic point, we have that there is a neighbourhood of $\xi_i$ that we denote by
$V(\xi_i)$ where $g|_{V(\xi_i)}:V(\xi_i)\to g(V(\xi_i))\subset\mathcal X$ is a diffeomorphim.
By the Hartman-Grobman theorem there is a neighbourhood $W(0)$ of $0\in T_{\zeta}\mathcal X$
and a homeomorphism $h:W(0)\to V(\zeta)$ such that $h\circ Dg(\zeta)=g\circ h$.

 As $W(0)$ is an open set, there is $r>0$ so that the ball of radius $r$,
 denoted by $B_r(0)$ is contained in $W(0)$. Let $x\in B_r(\xi_i)$ and let $\epsilon=\dist(x,\xi_i)$.
 By hypothesis, the point $\xi_i $ is repelling periodic and it implies
  that there exists $\lambda>1$ such that the diameter of  $(Dg(\zeta))^\kappa(B_{\epsilon}(\xi_i))$ is
at least $\lambda^\kappa\epsilon$.  So, in order for $\lambda ^\kappa>r$, we must have that $\kappa>\frac{\log r-\log\epsilon}{\log \lambda}$.
Hence, $Dg(\zeta))^\kappa(B_{\epsilon}(\xi_i))\subset B_r(\xi_i)$,
 for $\kappa\leq\frac{\log r-\log\epsilon}{\log \lambda}$.

Now, for each $j\in\{1,...,N\}$, let $V(\xi_j)$ be an open neighbourhood of
$\xi_j$  where $f|_{V(\xi_j)}:V(\xi_j)\to V(f^{m_{j+1}}(\zeta))$
is a diffeomorphism and the Hartman-Grobman theorem holds.
Let $n\in \N$ such that
$U_n(\xi_j)\subset V(\xi_j)$.
Given $x\in \AA_{q,n}$, by the definition of $\AA_{q,n}$ we know that $g^\kappa(x)\notin \AA_{q,n}$ while $g^\kappa(x)\in B_r(\xi_i)$. Let $\epsilon_n$ be the diameter of $\AA_{q,n}$. Then while
$\kappa\leq\frac{\log r-\log \epsilon_n}{\log \lambda}$ we have that $g^\kappa(x)\notin \AA_{q,n}$. Since $\frac{\log r-\log \epsilon_n}{\log \lambda}\to\infty$, as $n\to\infty$, then it follows that $\lim_{n\to\infty}R(\AA_{q,n})=\infty$.
\end{proof}

\subsection{The extremal index}
In what follows we need some notation, which we introduce now:
We define $m_{i+N}:=m_i+p$, so that $f^{m_{i+N}}(\zeta)=f^{m_{i}}(\zeta)$.

\begin{proposition}\label{prop:extremal-index-period}
Let $X_0,X_1,...$ be given by \eqref{eq:def-stat-stoch-proc-DS}, where the observable
$\varphi$ is given by \eqref{def:observable}, with $\zeta\in \mathcal X$ a repelling periodic
point. In this case
the extremal index $\theta$ is given by
$$
\theta=\lim_{n\to\infty} \frac{\displaystyle\sum_{i=1}^N \mu\left(U_{n}^{(0)}(\xi_i)-\displaystyle\bigcup_{j=i+1}^{N+i} f^{-(m_j-m_i)}(U_{n}^{(0)}(\xi_j))\right)}
{\displaystyle\sum_{i=1}^N \mu\left(U_{n}^{(0)}(\xi_i)\right)},
$$
whenever the limit exists.
\end{proposition}
\begin{proof} Note that $\AA_{p,n}^{(0)}=\bigcup_{i=1}^N \AA_{p,n}^{(0)}(\xi_i)$, where
$$\AA_{p,n}^{(0)}(\xi_i)=U_{n}^{(0)}(\xi_i)-\bigcup_{j=i+1}^{N+i} f^{-(m_j-m_i)}(U_{n}^{(0)}(\xi_j)).$$
Since
$\displaystyle\theta=\lim_{n\to\infty} \frac{\mu(\AA_{p,n}^{(0)})}{\mu(U_n^{(0)})}$, the conclusion follows.
\end{proof}

Just like in the non-peridic case, we now consider some particular examples that illustrate the richness of scenarios that the formula for the EI of the previous proposition encompasses. 
Hence, as before, we assume that condition ($R3$) holds.

First, we introduce some notation.

Let $I_1$ be the set of indices $i\in  \{1,\dots N\}$ such that
\begin{equation}\label{eq:I1p} \bigcup_{\ell=i+1}^{N+i} f^{-(m_l-m_i)}(U_{n}^{(0)}(\xi_\ell)) \cap\partial U_{n}^{(0)}(\xi_i)=\emptyset.\end{equation}
Under ($R3$), for every $i\in I_1$, there exists $j_{i}\in \{i+1,\dots N+i\}$ such that
\begin{equation*}
\bigcup_{\ell=i+1}^{N+i} f^{-(m_l-m_i)}(U_{n}^{(0)}(\xi_l))\cap U_{n}^{(0)}(\xi_i)=f^{-(m_{j_{i}}-m_i)}(U_{n}^{(0)}(\xi_{j_{i}}))\cap U_{n}^{(0)}(\xi_i).
\end{equation*}

\begin{corollary}\label{particular3}
If we are under the conditions of Proposition~\ref{prop:extremal-index-period} and the
condition $(R3)$ is satisfied,
 then the extremal index $\theta$ is given by
$$
\theta=\lim_{n\to\infty} \frac{\sum_{i\in I_1} \left[\mu\left(U_{n}^{(0)}(\xi_i)\right)- \mu\left(f^{-(m_{j_i}-m_i)}(U_{n}^{(0)}(\xi_{j_i}))\cap U_{n}^{(0)}(\xi_i)\right)\right]}
{\sum_{i=1}^N \mu\left(U_{n}^{(0)}(\xi_i)\right)},
$$
whenever the limit exists.
\end{corollary}

Let us now assume we are in the particular case where $\mu$ is absolutely continuous with respect to the Lebesgue measure and \eqref{eq:density} holds.

\begin{corollary} If we are under the hypotheses of Corollary \ref{particular3}  and $\mu$ is absolutely continuous and \eqref{eq:density} holds then
the extremal index $\theta$ is given by
$$
\theta=\lim_{n\to\infty} \frac{\displaystyle\sum_{i\in I_1} \left[\mu\left(U_{n}^{(0)}(\xi_i)\right)- \displaystyle\frac{1}{\left|\det Df^{m_{j_i}-m_i}(\xi_{i})\right|}\mu\left(U_{n}^{(0)}(\xi_{j_i})\right)\right]}
{\displaystyle\sum_{i=1}^N \mu\left(U_{n}^{(0)}(\xi_i)\right)}.
$$
\end{corollary}
\begin{proof}
The proof is analogous to the proof of the non-periodic case presented in Corollary~\ref{particular1cor}.
\end{proof}

\subsection{The multiplicity distribution}

As mentioned before, by Proposition 2.10 of \cite{FFR15} we have that the multiplicity
distribution of the limiting compound Poisson process for the REPP is given by the formula \eqref{def:mult-distr},
whenever the limit exists.

Recall that we are assuming that ($R3$) holds.

Let $I_1$ be defined as in \eqref{eq:I1p}. For every $i\in \{1,\dots,N\}$, we define $j_{i,0}=i$. Then, for $i\in \{1,\dots,N\}$ we inductively define $I_k$ in the following way.

Assuming that $I_k$ is defined as well as $j_{1,k-1}, \ldots, j_{N,k-1}$, then for every $i=1, \ldots, N$, we define 
$j_{i,k}$ to the index such that
\begin{equation}
\bigcup_{\ell=i+1}^{N+i} f^{-(m_{j_{\ell,k-1}}-m_i)}(U_{n}^{(0)}(\xi_{j_{\ell,k-1}}))\cap U_{n}^{(0)}(\xi_i)=f^{-(m_{j_{i,k}}-m_i)}(U_{n}^{(0)}(\xi_{j_{i,k}}))\cap U_{n}^{(0)}(\xi_i),
\end{equation}
when $i\in I_{k}$ and $j_{i,k}=i$, otherwise.

Now, we define $I_{k+1}$ as the set of indices $i\in  \{1,\dots N-k\}$ such that
$$ \bigcup_{\ell=i+1}^{N+i} f^{-(m_{j_{\ell,k}}-m_i)}(U_{n}^{(0)}(\xi_{j_{\ell,k}}))\cap \partial U_{n}^{(0)}(\xi_i)=\emptyset.$$

\begin{proposition}
\label{prop:multip-R3-per}
The multiplicity distribution of the point process is given by equation \eqref{def:mult-distr-per} (whenever the limit exists),
where, for $k<N$,
\begin{align*}
\mu(\AA^{(k)}_{q,n})=
\sum_{i\in I_{k+1}} &\left[\mu\left(f^{-(m_{j_{i,k}}-m_i)}(U_{n}^{(0)}(\xi_{j_{i,k}})) \cap U_{n}^{(0)}(\xi_i)\right)\right.\\
&\left.-\mu\left(f^{-(m_{j_{i,k+1}}-m_i)}(U_{n}^{(0)}(\xi_{j_{i,k+1}}))\cap U_{n}^{(0)}(\xi_i)\right)\right]
\end{align*}
 and, for $k\geq N$, $\mu(\AA^{(k)}_{q,n})=0$.

\end{proposition}

\begin{proof}
We will prove by induction that, for $i\in I_{k+1}$,
$$\AA_{q,n}^{(k)}(\xi_i)= U_{n}^{(0)}(\xi_i)\cap f^{-(m_{j_{i,k}}-m_i)}(U_{n}^{(0)}(\xi_{j_{i,k}}))-f^{-(m_{j_{i,k+1}}-m_i)}(U_{n}^{(0)}(\xi_{j_{i,k+1}}))$$
and $$U_{n}^{(k+1)}(\xi_i)=f^{-(m_{j_{i,k+1}}-m_i)}(U_{n}^{(0)}(\xi_{j_{i,k+1}}))\cap U_{n}^{(0)}(\xi_i),$$
and, for $i\not\in I_{k+1}$, $\mu\left(\AA_{q,n}^{(k)}(\xi_i)\right)=0$ since $U_{n}^{(k)}(\xi_i)=U_{n}^{(k-1)}(\xi_i)$.

For $i\in I_1$ since $j_{\ell,0}=\ell$, we get
\begin{align*}
\AA_{q,n}^{(0)}(\xi_i)&=U_{n}^{(0)}(\xi_i)-\bigcup_{\ell=i+1}^{N+i} f^{-(m_\ell-m_i)}(U_{n}^{(0)}(\xi_\ell))
\\&=U_{n}^{(0)}(\xi_i)-\bigcup_{\ell=i+1}^{N+i} f^{-(m_{j_{\ell.0}}-m_i)}(U_{n}^{(0)}(\xi_{j_{\ell,0}}))
\\&=U_{n}^{(0)}(\xi_i)\cap f^{-(m_{j_{i,0}}-m_i)}(U_{n}^{(0)}(\xi_{j_{i,0}}))-f^{-(m_{j_{i,1}}-m_i)}(U_{n}^{(0)}(\xi_{j_{i,1}})),
\end{align*}
by definition of $j_{i,1}$, and, moreover, we have $U_{n}^{(1)}(\xi_i)=f^{-(m_{j_{i,1}}-m_i)}(U_{n}^{(0)}(\xi_{j_{i,1}}))\cap U_{n}^{(0)}(\xi_i)$.

For $i\not\in I_1$, $i\not =N$, $U_{n}^{(0)}(\xi_i)-\bigcup_{\ell=i+1}^{N+i} f^{-(m_\ell-m_i)}(U_{n}^{(0)}(\xi_l))=\emptyset$, so $\mu\left(\AA_{q,n}^{(0)}(\xi_i)\right)=0$ and $\mu\left(U_{n}^{(1)}(\xi_i)\right)=\mu\left(U_{n}^{(0)}(\xi_i)\right)$.

By induction, for $i\in I_{k+1}$, we obtain
\begin{align*}
\AA_{q,n}^{(k)}(\xi_i)&=U_{n}^{(k)}(\xi_i)-\bigcup_{\ell=i+1}^{N+i} f^{-(m_\ell-m_i)}(U_{n}^{(k)}(\xi_l))
\\&=U_{n}^{(k)}(\xi_i)-\bigcup_{\ell=i+1}^{N+i} f^{-(m_\ell-m_i)}\left(f^{-(m_{j_{\ell,k}}-m_l)}(U_{n}^{(0)}(\xi_{j_{\ell,k}}))\cap U_{n}^{(0)}(\xi_i)\right),\,\mbox{by induction}
\\&=f^{-(m_{j_{i,k}}-m_i)}(U_{n}^{(0)}(\xi_{j_{i,k}})) \cap U_{n}^{(0)}(\xi_i)-\mu\left(\bigcup_{\ell=i+1}^{N+i} f^{-(m_{j_{\ell,k}}-m_i)}(U_{n}^{(0)}(\xi_{j_{\ell,k}}))\right), \, \mbox{by induction}
\\&=f^{-(m_{j_{i,k}}-m_i)}(U_{n}^{(0)}(\xi_{j_{i,k}}))\cap U_{n}^{(0)}(\xi_i)-f^{-(m_{j_{i,k+1}}-m_i)}(U_{n}^{(0)}(\xi_{j_{i,k+1}})), \mbox{ by definition of $j_{i,k+1}$}.
\end{align*}
and in particular we also have  $U_{n}^{(k+1)}(\xi_i)=f^{-(m_{j_{i,k+1}}-m_i)}(U_{n}^{(0)}(\xi_{j_{i,k+1}}))\cap U_{n}^{(0)}(\xi_i)$.

By induction, for $i\in I_{k+1}$, we obtain
\begin{align*}
&\AA_{q,n}^{(k)}(\xi_i)=U_{n}^{(k)}(\xi_i)-\bigcup_{\ell=i+1}^{N+i} f^{-(m_\ell-m_i)}(U_{n}^{(k)}(\xi_l))
\\&=U_{n}^{(k)}(\xi_i)-\bigcup_{\ell=i+1}^{N+i} f^{-(m_\ell-m_i)}\left(f^{-(m_{j_{\ell,k}}-m_l)}(U_{n}^{(0)}(\xi_{j_{\ell,k}}))\cap U_{n}^{(0)}(\xi_i)\right),\,\mbox{by induction}
\\&=f^{-(m_{j_{i,k}}-m_i)}(U_{n}^{(0)}(\xi_{j_{i,k}})) \cap U_{n}^{(0)}(\xi_i)-\mu\left(\bigcup_{\ell=i+1}^{N+i} f^{-(m_{j_{\ell,k}}-m_i)}(U_{n}^{(0)}(\xi_{j_{\ell,k}}))\right), \, \mbox{by induction}
\\&=f^{-(m_{j_{i,k}}-m_i)}(U_{n}^{(0)}(\xi_{j_{i,k}}))\cap U_{n}^{(0)}(\xi_i)-f^{-(m_{j_{i,k+1}}-m_i)}(U_{n}^{(0)}(\xi_{j_{i,k+1}})), \mbox{ by definition of $j_{i,k+1}$}.
\end{align*}
and in particular we also have  $U_{n}^{(k+1)}(\xi_i)=f^{-(m_{j_{i,k+1}}-m_i)}(U_{n}^{(0)}(\xi_{j_{i,k+1}}))\cap U_{n}^{(0)}(\xi_i)$.

For $i\not\in I_{k+1}$, $i\not =N-k+1$ $\mu\left(\AA_{q,n}^{(k)}(\xi_i)\right)=0$ and
$$U_{n}^{(k+1)}(\xi_i)=f^{-(m_{j_{i,k}}-m_i)}(U_{n}^{(0)}(\xi_{j_{i,k}}))\cap U_{n}^{(0)}(\xi_i).$$

This concludes the proof.
\end{proof}

\begin{corollary}\label{particular4}Assume that $\mu$ is absolutely continuous with respect to the Lebesgue measure and \eqref{eq:density} holds. Then the multiplicity distribution of the point process is given by formula \eqref{def:mult-distr},
where, for $k<N$,
\begin{align*}
\mu(\AA^{(k)}_{q,n})=\displaystyle\sum_{i\in I_{k+1}} \Big(\alpha_n(m_{j_{i,k}},m_i)-\alpha_n(m_{j_{i,k+1}},m_i)\Big)+\alpha_n(m_{j_{N-k,k}},m_{N-k})
\end{align*}
and
$$\alpha_n(a,b)=\displaystyle\frac{1}{|\det Df^{a-b}(f^b(\zeta))|}\mu \left(U_{n}^{(0)}(f^a(\zeta))\right).$$
\end{corollary}
\begin{proof}
The proof is analogous to the proof of the non-periodic case presented in Corollary~\ref{particular2}.
\end{proof}

\begin{corollary}\label{particular5}
Assume that $\mu$ is absolutely continuous with respect to the Lebesgue measure and \eqref{eq:density} holds and
$$ f^{-(m_{i+1}-m_i)}\left(U_{n}^{(0)}(\xi_{i+1})\right) \subset U_{n}^{(0)}(\xi_i).$$
Then the multiplicity distribution of the point process is given by formula \eqref{def:mult-distr}
where, for $k<N$,
\begin{align*}
\mu(\AA^{(k)}_{q,n})=\sum_{i\in I_{k+1}} \Big(\alpha_n(m_{i+k},m_i)-\alpha_n(m_{i+k+1},m_i)\Big)+\alpha_n(m_N,m_{N-k})
\end{align*}
and $\alpha_n(a,b)$ is as in Corollary~\ref{particular4}.
\end{corollary}
\begin{proof}
Using the notation of Proposition~\ref{prop:multip-R3-per}, note that in this case $j_{i,k}=i+k$.
\end{proof}

\subsubsection{Example}

Let $(f,\mathbb{S}^1,Leb,\mathcal{B})$ be the dynamical system where
 $f:\mathbb{S}^1\to \mathbb{S}^1$ is the map given by $f(x)=2x \mod 1$. Take the periodic point of period $5$, $\zeta= \frac{1}{31}$, and an observable $\varphi: \mathbb{S}^1\to \mathbb R\cup \{\infty\}$
such that
$$
\varphi(x)=\left\{
             \begin{array}{lll}
               -\log|x-\zeta|, & \hbox{ if }x\in B_\epsilon(\zeta), \\
               |x-f(\zeta)|^{\frac{-1}{2}}, & \hbox{ if }x\in B_\epsilon(f(\zeta)) \\
               |x-f^3(\zeta)|^{\frac{-1}{2}}, & \hbox{ if }x\in B_\epsilon(f^3(\zeta))\\
               0, & \hbox{ otherwie},
             \end{array}
           \right.
$$
where we choose $\epsilon>0$, such that $ B_\epsilon(f^i(\zeta))\cap B_\epsilon(f^j(\zeta))=\emptyset$
for $ i,j\in\{0,1,3\} $ and $i\not=j$. In this case
we have that $u_F=\sup_{x\in  \mathbb{S}^1} \varphi(x)=\infty$ and it occurs in $\zeta, f(\zeta)$ and $f^3(\zeta)$.

Given $\tau>0$ and a sequence $(u_n)_{n\in\mathbb N}$ satisfying \ref{eq:un} we have that
$$
U_n=(\zeta-e^{-u_n},\zeta+e^{-u_n})\cup(f(\zeta)-u_n^{-2},f(\zeta)+u_n^{-2})
\cup(f^3(\zeta)-u_n^{-2},f^3(\zeta)+u_n^{-2}).
$$
We notice that for $\tau$ sufficiently large we have that
$$
(\zeta-e^{-u_n},\zeta+e^{-u_n})\varsubsetneqq f^{-1}\left((f(\zeta)-u_n^{-2},f(\zeta)+u_n^{-2})\right),
$$
$$
f^{-4}(\zeta-e^{-u_n},\zeta+e^{-u_n})\varsubsetneqq f^{-2}\left((f^3(\zeta)-u_n^{-2},f^3(\zeta)+u_n^{-2})\right)\varsubsetneqq \left(f(\zeta)-u_n^{-2},f(\zeta)+u_n^{-2}\right),
$$
and
$$
f^{-2}(\zeta-e^{-u_n},\zeta+e^{-u_n})\varsubsetneqq f^{-3}\left((f(\zeta)-u_n^{-2},f(\zeta)+u_n^{-2})\right)\varsubsetneqq \left(f^3(\zeta)-u_n^{-2},f^3(\zeta)+u_n^{-2}\right).
$$
\vspace{5mm}
\textbf{The extremal index:}
Noting that $m_1=0,m_2=1, m_3=3$, then $I_0=\{2,3\}$ and $j_2=3, j_3=5$. So, by Corollary \ref{particular4} we have that

\begin{align*}
\theta=\lim_{n\to\infty} &\left( \frac{\left[\mu\left(U_{n}^{(0)}(\xi_i)\right)- \displaystyle\frac{1}{Df^{m_3-m_2}(f^{m_3}(\zeta))}\mu\left(U_{n}^{(0)}(f^{m_3}(\zeta))\right)\right]}
{\sum_{i=1}^N \mu\left(U_{n}^{(0)}(\xi_i)\right)}\right.
\\&\left.+\frac{\left[\mu\left(U_{n}^{(0)}(\xi_i)\right)- \displaystyle\frac{1}{Df^{m_5-m_3}(f^{m_5}(\zeta))}\mu\left(U_{n}^{(0)}(f^{m_5}(\zeta))\right)\right]}
{\sum_{i=1}^N \mu\left(U_{n}^{(0)}(\xi_i)\right)}\right)
\\&=\lim_{n\to\infty}\frac{u_n^{-2}-\displaystyle\frac{1}{Df^{2}(f^3(\zeta))}u_n^{-2}+u_n^{-2}-\displaystyle\frac{1}{Df^{3}(f^2(\zeta))}u_n^{-2}}{e^{-u_n}+u_n^{-2}+u_n^{-2}}
=\frac{2-\frac{1}{4}-\frac{1}{8}}{2}=\frac{13}{16}.
\end{align*}
\vspace{5mm}
\textbf{The multiplicity distribution:}
To compute the multiplicity we will use Corollary \ref{particular5}. First, we need to determine the sets $I_k$, for
all $k\in \mathbb{N}$. It is not difficult to see that, in this case,
$I_k=\{2,3\}$, for all $k\in \mathbb{N}$. We now need to determine the value of $j_{i,k}$, for all $k\in \mathbb{N}$ and all $i\in I_k$. By definition $j_{i,0}=i$, for all $i\in\mathbb{N}$.
We can verify by induction that
\begin{align*}
&j_{2,2\ell}=3\ell-2, \quad j_{2,2\ell-1}=3\ell-3,
\\& j_{3,2\ell}=3\ell-3, \quad j_{3,2\ell-1}=3\ell-5.
\end{align*}

In order to compute $\AA^{(k)}_{q,n}$, note that in this case $q=5$. So,
\begin{align*}
\mu(\AA^{(k)}_{5,n})&=\displaystyle\sum_{i\in I_{k+1}} \Big(\alpha_n(m_{j_{i,k}},m_i)-\alpha_n(m_{j_{i,k+1}},m_i)\Big)
\\&=\alpha(m_{j_{2,k}},m_2)-\alpha(m_{j_{2,k+1}},m_2)+\alpha(m_{j_{3,k}},m_3)-\alpha(m_{j_{3,k}},m_3).
\end{align*}
Considering first the case $k=2\ell$, with $\ell\in\mathbb{N}$, we have
\begin{align*}
\mu(\AA^{(2\ell)}_{5,n})&=\alpha(m_{j_{2,2\ell}},m_2)-\alpha(m_{j_{2,2\ell-1}},m_2)+\alpha(m_{j_{3,2\ell}},m_3)-\alpha(m_{j_{3,2\ell-1}},m_3)
\\&=\alpha(m_{3\ell-2},m_2)-\alpha(m_{3\ell-3},m_2)+\alpha(m_{3\ell-3},m_3)-\alpha(m_{3\ell-5},m_3)
\\&=\left(\frac{1}{2}\right)^{5\ell}u_n^{-2}-\left(\frac{1}{2}\right)^{5\ell-2}u_n^{-2}+\left(\frac{1}{2}\right)^{5\ell}u_n^{-2}-\left(\frac{1}{2}\right)^{5\ell-3}u_n^{-2}=\frac{13}{8}\left(\frac{1}{2}\right)^{5\ell}u_n^{-2}.
\end{align*}
For the case $k=2\ell+1$ with $\ell\in\mathbb{N}_0$, we have
\begin{align*}
\mu(\AA^{(2\ell+1)}_{5,n})&=\alpha(m_{j_{2,2\ell-1}},m_2)-\alpha(m_{j_{2,2\ell-2}},m_2)+\alpha(m_{j_{3,2\ell-1}},m_3)-\alpha(m_{j_{3,2\ell-2}},m_3)
\\&=\alpha(m_{3\ell-3},m_2)-\alpha(m_{3\ell-5},m_2)+\alpha(m_{3\ell-5},m_3)-\alpha(m_{3\ell-6},m_3)
\\&=\left(\frac{1}{2}\right)^{5\ell-2}u_n^{-2}-\left(\frac{1}{2}\right)^{5\ell-5}u_n^{-2}+\left(\frac{1}{2}\right)^{5\ell-3}u_n^{-2}-\left(\frac{1}{2}\right)^{5\ell-5}u_n^{-2}=\frac{5}{16}\left(\frac{1}{2}\right)^{5\ell}u_n^{-2}.
\end{align*}
Replacing these in the expression in Corollary \ref{particular4} we obtain
\begin{align*}
&\pi(2\ell+1)=\frac{\mu(\AA^{(2\ell)}_{5,n})-\mu(\AA^{(2\ell+1)}_{5,n})}{\mu(\AA^{(0)}_{5,n})}=\frac{\displaystyle\frac{13}{8}\left(\frac{1}{2}\right)^{5\ell}u_n^{-2}-\displaystyle\frac{5}{16}\left(\frac{1}{2}\right)^{5\ell}u_n^{-2}}{\displaystyle\frac{13}{8}u_n^{-2}}=\frac{21}{26}\left(\frac{1}{2}\right)^{5\ell}, \mbox{for } \ell\in\mathbb{N}_0,
\\&\pi(2\ell)=\frac{\mu(\AA^{(2\ell-1)}_{5,n})-\mu(\AA^{(2\ell)}_{5,n})}{\mu(\AA^{(0)}_{5,n})}=\frac{\displaystyle\frac{5}{16}\left(\frac{1}{2}\right)^{5(\ell-1)}u_n^{-2}-\frac{13}{8}\left(\frac{1}{2}\right)^{5\ell}u_n^{-2}}{\displaystyle\frac{13}{8}u_n^{-2}}=\frac{67}{13}\left(\frac{1}{2}\right)^{5\ell}, \mbox{for } \ell\in\mathbb{N}.
\end{align*}

\subsection{Non correlated multiple maxima including periodic points}
\label{subsec:non-correlated-periodic}

We discuss briefly the possibility of having multiple maximal points with non-intersecting orbits as in Section~\ref{subsec:non-correlated-non-periodic} but now allowing the possibility of having maximal points that are periodic. By \cite[Theorem~2]{FFT13} and \cite[Theorem~A]{AFV14}, we have that for every continuous system under Assumption A, for every $\zeta\in\X$ taken as the only maximum of $\varphi$, then there exists an EVL with an EI that is 1 if the point is not periodic or is less than 1 (whose value depends on the expansion rate at $\zeta$). Hence, we are going to consider multiple maximal points $\xi_1, \ldots, \xi_N$ and to each such point we associate the respective EI guaranteed by  \cite[Theorem~2]{FFT13} and \cite[Theorem~A]{AFV14} (see also \cite[Theorem 8]{F13}), which we denote by $\theta_i$, being that if $\xi_i$ is periodic then $\theta_i<1$.

\begin{proposition}
Assume that $f$ is continuous and satisfies Assumption A. Let $X_0,X_1,...$ be given by \eqref{eq:def-stat-stoch-proc-DS}, where the observable
$\varphi$ has $N$ maximal points $\xi_1,\ldots,\xi_N$ and let $\theta_i$ be the corresponding EI. Let $p_i$ be the period of $\zeta_i$ (we set $p_i=0$ when $\zeta_i$ is not periodic). Assume that \eqref{eq:non-intersection} holds and that the following limits exist for every $i=1,\ldots,N$
$$
\lim_{n\to\infty}\frac{\mu(U_n(\xi_i))}{\mu(U_n)}=:\alpha_i,
$$
where $U_n(\xi_i)$ is a connected component of $U_n$ as in \eqref{eq:connected-components-U}.
Then $X_0,X_1,...$ satisfies conditions $\D_q(u_n)$ and $\D'_q(u_n)$, for $q=\max_{i=1,\ldots,N} p_i$, and consequently for $(u_n)_{n\in\N}$ satisfying \eqref{eq:un} then $\lim_{n\to\infty}\mu(M_n\leq u_n)=\e^{-\theta\tau}$, where $\theta=\sum_{i=1}^N \alpha_i \theta_i.$
\end{proposition}

\begin{proof}
Conditions $\D_q(u_n)$ and $\D'_q(u_n)$ follow from the computations in Section~\ref{subsec:existence} as long as one check that $\lim_{n\to\infty}R(\AA_{q,n})=+\infty$. This last statement follows from the definition of $\AA_{q,n}$ and a continuity argument very similar to the proof of Lemma~\ref{lem:continuity}.
\end{proof}

\section{Clustering patterns}

As we have seen in Sections~\ref{sec:nonperiodic} and \ref{sec:periodic}, multiple correlated maxima can be used as a mechanism to create clustering, alter the value of EI and produce different multiplicity distributions for the limit of REPP. When compared to the usual method of introducing clustering in a dynamical setting, which was essentially based on considering a global maximum of $\varphi$ achieved at a unique repelling periodic point, this method of forging clustering by means of considering multiple correlated maxima is much richer.

In fact, we have already seen that, on the contrary to single maximum at a repelling periodic point for which one always obtains a geometric distribution for the multiplicity distribution of the limiting compound Poisson process for the REPP, with multiple correlated maxima we can obtain different multiplicity distributions for the limit of the REPP.

But the richness of this mechanism can also be appreciated by simply looking at the observed data and the corresponding clustering patterns. Observe that the systems we are studying are either uniformly expanding or non-uniformly expanding, which means that the periodic points should be repelling. When the maximum is achieved at a single repelling point the clustering pattern is very simple: it consists of a very large observation corresponding to the first exceedance, which is followed by a strictly decreasing cluster of exceedances observed after a precise number of observations corresponding to the period (see Figure \ref{fig:old-cluster}). This is because in order to have a cluster of exceedances the orbit needs to enter in a very small neighbourhood of the periodic point $\zeta$ of period $p$. Let us assume we have a very high exceedance, which of course means that the orbit went very close to $\zeta$. Once inside this very small neighbourhood of $\zeta$, the periodicity of $\zeta$ forces the appearance of another exceedance after $p$ observations but since $\zeta$ is repelling this return to a small neighbourhood of $\zeta$ is not that deep, \ie the orbit is pulled away from $\zeta$ which explains the fading of the exceedances.

\begin{figure}[!h]
\begin{center}
\includegraphics[width=12cm]{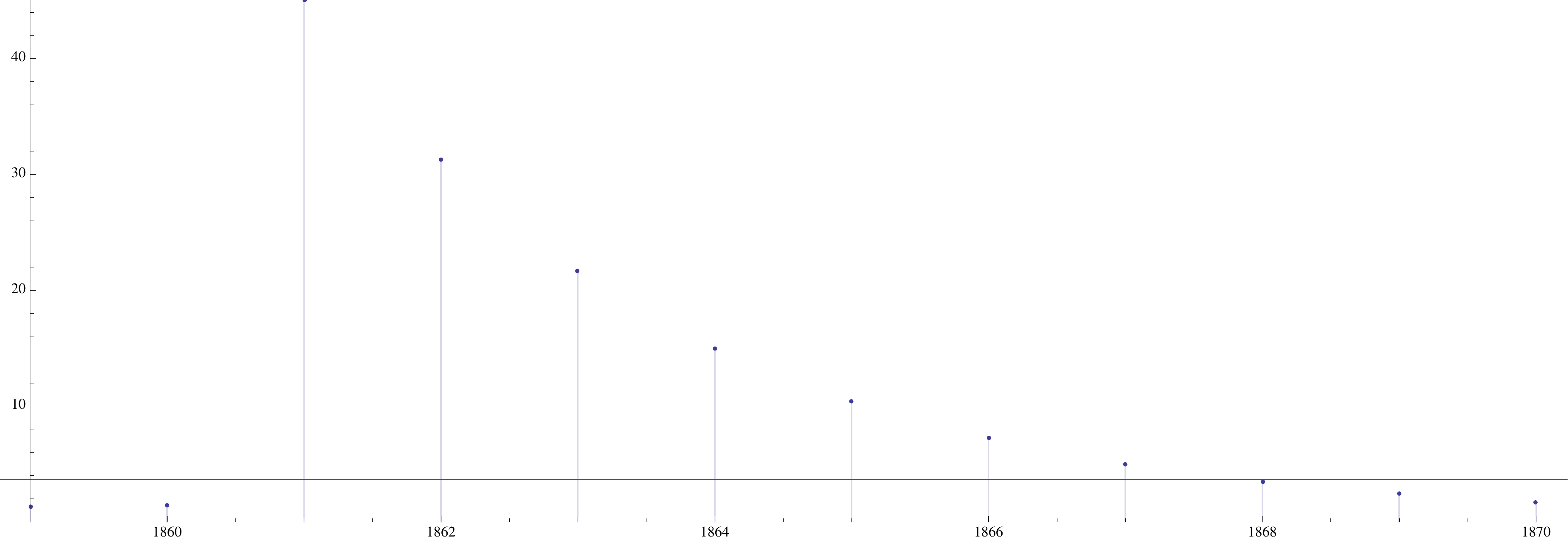}
\caption{Cluster of observations for an observable with global maximum achieved at the fixed point $1/2$ of the system $f$ given by \eqref{eq:system-example}}
\label{fig:old-cluster}
\end{center}
\end{figure}

In order to illustrate how easily one can create a different clustering pattern, where in particular one can observe that the exceedances may grow inside a cluster, we consider the following example.

We consider the system uniformly expanding map
\begin{align}
\label{eq:system-example}
f:S^1&\longrightarrow S^1
\\x &\mapsto 3x \mod 1\nonumber
\end{align}
and the observable $\varphi:S^1\longrightarrow \mathbb{R}$, where $\varphi(x)=-\log|x-\frac{1}{4}| \textbf{1}_{[0,\frac{1}{2}]}+|x-\frac{3}{4}|^{-\frac{1}{3}} \textbf{1}_{(\frac{1}{2},1]}$. This observable the maximum value is $u_F=\sup_{x\in  \mathbb{S}^1} \varphi(x)=\infty$, which is achieved at the points $\zeta=\frac{1}{4}$ and $f(\zeta)=\frac{3}{4}$, being that $\zeta=f^2(\zeta)$, \ie, the maxima correspond to the points of a periodic orbit of period 2.

Note that for $u$ sufficiently large, $\{X_0>u\}=(1/4-\e^{-u},1/4+\e^{-u})\cup(3/4-u^{-3},3/4+u^{-3})$. So an exceedance of $u$ occurs when the orbit enters a ball of radius $\e^{-u}$ around $1/4$ or a ball of radius $u^{-3}$ around $3/4$. Note that if we have an exceedance because the orbit hits the point $x=1/4+\e^{-u}/3$ and $\varphi(x)=u+\log3>u$, then in the next iterate the orbit hits the point $f(x)=3/4+e^{-u}$ and $\varphi(f(x))=\e^{u/3}\gg u+\log3>u$, for $u$ sufficiently large so that we have an exceedance that is followed by another exceedance that is even higher than the first.

This way, we create different clustering patterns which can be fully appreciated by the simulation study we performed. We chose a random point $x\in [0,1]$ (using the uniform distribution); we iterated the point 2000 times by $f$; we considered $u$ to be such that $\e^{-u}+\frac{1}{u^3} = \frac{40}{2000}$ (meaning that the expected number of exceedances of $u$ is 40). In Figure~\ref{fig:timeseries1} we can see the actual time series and in Figure~\ref{fig:new-cluster} we can see the observations forming a cluster that presents clearly a different pattern, with exceedances being followed by even higher exceedances, which was forbidden in the case of a single maximum as portrayed in Figure~\ref{fig:old-cluster}.

\begin{figure}[!h]
\begin{center}
\includegraphics[width=\textwidth]{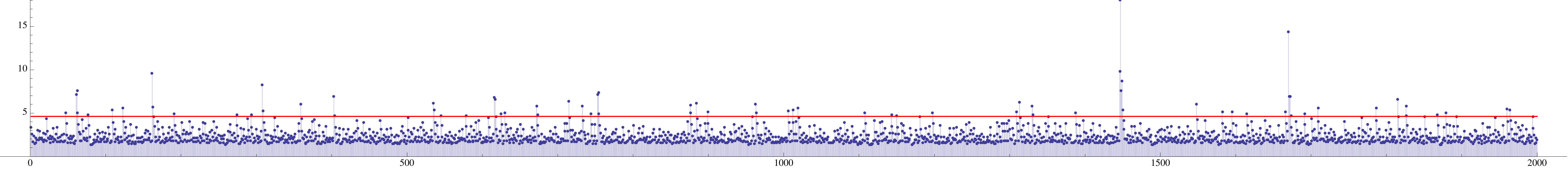}
\caption{Time series for 2000 runs of the orbit of the point $x=0.7756592465669858$ with the level $u=4.619613119957849$ depicted in red}
\label{fig:timeseries1}
\end{center}
\end{figure}

\begin{figure}[!h]
\begin{center}
\includegraphics[width=10cm]{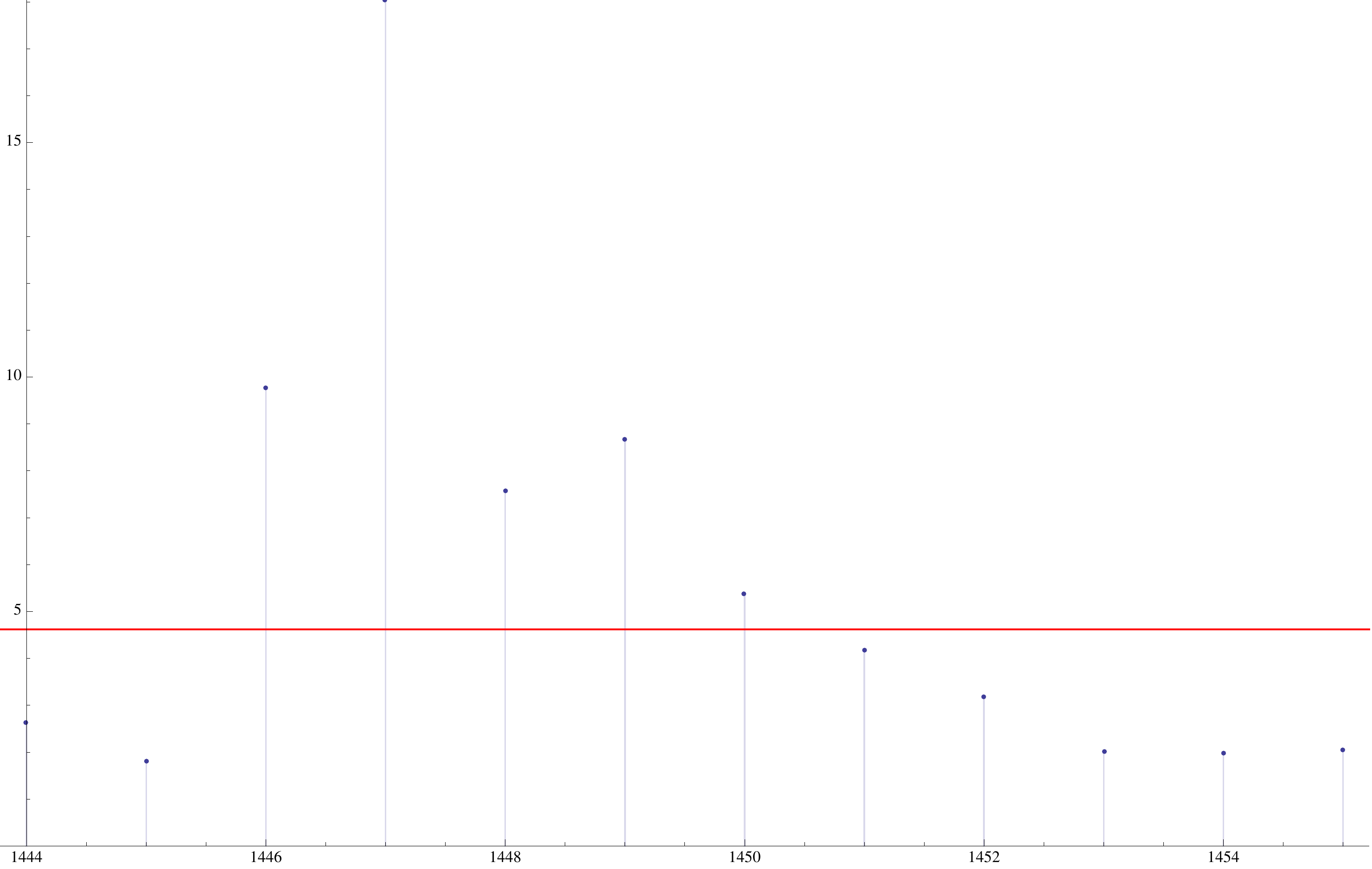}
\caption{Cluster occurring between the observations 1446 and 1450 of the time series in Figure~\ref{fig:timeseries1}}
\label{fig:new-cluster}
\end{center}
\end{figure}


\section{Competition between domains of attraction}

Up until now, we have addressed the issues regarding the effect of having multiple maxima on the existence of EVL, the appearance of clustering, the consequences on the value of the  EI and the impact on the limit of  REPP. This has been done for normalising sequences $(u_n)_{n\in\N}$ such that condition \eqref{eq:un} holds. Typically, these normalising sequences are taken as a one parameter family of linear sequences such that $u_n=y/a_n+b_n$, where $(a_n)_{n\in\N}$, $(b_n)_{n\in\N}$, with $a_n>0$, and one usually looks for non-degenerate limit distributions for $\p(a_n(M_n-b_n)\leq y)$.

The limit is then $\e^{-\tau(y)}$, where $\tau(y)$ is a function of $y$, which is connected to the type of observable we have. For well behaved measures, if the observable satisfies \eqref{eq:type1} then one gets that $\tau(y)=\e^{-y}$ meaning that the the stochastic process falls in the domain of attraction of a Gumbel law (or type 1); if the observable satisfies \eqref{eq:type2} then one gets that $\tau(y)=y^{-\alpha}$ (for $y>0$ and $\alpha>0$) meaning that the the stochastic process falls in the domain of attraction of a Fr\'echet law (or type 2); observable satisfies \eqref{eq:type3} then one gets that $\tau(y)=(-y)^\alpha$ (for $y\geq 0$ and $\alpha>0$) meaning that the the stochastic process falls in the domain of attraction of a Weibull law (or type 3).

In the independent and identically distributed (i.i.d.) setting the domain of attraction is determined by the tail of the distribution $F$, where $F(u)=\mu(X_0\leq u)$. As it was observed in \cite[Remark~1]{FFT10}, in the dynamical setting (with a single maximum), the shape of the observable determines the type of tail of the distribution $F$ and, once the existence is established by the assumptions on the system, then the domain of attraction is determined by the behaviour of the observable at the maximum.

In here, since we have multiple maxima, then we may have different types of behaviour of the observable at the different maximal points. This creates a sort of competition between different domains of attraction, which in fact is not that difficult to settle. First, we start by noting that if the observable is of type 2 (satisfies \eqref{eq:type2}) then $\sup \varphi=h(0)=\infty$;  if the observable is of type 3 (satisfies \eqref{eq:type3}) then $\sup \varphi=h(0)=D<\infty$ and if the observable is of type 1 (satisfies \eqref{eq:type1}) then the maximum value can be either finite or infinite. Hence, since all maxima achieve the same maximum value, we can only have competitions between Fr\'echet and Gumbel and between Weibull and Gumbel.

Note that the appearance of clustering and consequently of an EI less than 1 does not change the type of limit law that applies when we consider one parameter linear sequences. This is the statement of \cite[Corollary~3.7.3]{LLR83}, which also adds that, in fact, we can even perform a linear change of the normalising constants $a_n$ and $b_n$ so that the exact same limit applies as in the corresponding independent case. Hence in what follows we restrict the study to the behaviour of $\mu(U(u))=1-F(u)$ instead of $\mu(\A(u))$ as $u$ goes to $\sup\varphi$.

In what follows we establish the reigning domain of attraction in each such situation. For that purpose, we introduce the following notation $\bar{F}(x)=1-F(x)$.

\subsection{Gumbel versus Fr\'echet}

\begin{proposition}
Let $\varphi:\mathcal{X}\to\mathbb{R}\cup\{\infty\}$ be an observable
which has the following form
$$
\varphi(x)=\varphi_1(x)\I_{U(\zeta)}(x)+\varphi_2(x)\I_{U(f^i(\zeta))}(x),
$$
where $\varphi_1$ is an observable defined in the neighborhood $U(\zeta)$ of $\zeta$,
$\varphi_2$ is an observable defined in the neighborhood  $U(f^i(\zeta))$ of
$f^i(\zeta)$ and they satisfy
$\max_x\varphi_1(x)= \varphi_1(\zeta)=\infty=\max_x\varphi_2(x)= \varphi_2(f^i(\zeta)) $.
Let $F_i$ be the distribution function of $\varphi_i$. If $F_1 $ and $F_2$ belong to the domain of attraction for maxima of the Gumbel and the Frechet distribution, respectively, then the distribution function of  $\varphi $ belongs to the domain of attraction of the Frechet distribution.
\end{proposition}
\begin{proof}
We start by noting that if $F_1$ belongs to the domain of attraction of the Gumbel distribution, with infinite right endpoint of the support, then we may write
\begin{align}
1-F_1(u)=u^pe^{-u^q}L_1(u),
\end{align}
 where $ p\in \mathbb{R}$, $q>0$ and $L_1$ is a function of slow variation at infinity, i.e., $L_1:(0,+\infty)\to(0,+\infty)$ and,
for all $t>0$,
$$
\lim_{u\to\infty}\frac{L_1(tu)}{L_1(u)}=1
$$
(cf. \cite{G84}).

$F_2$ belongs to the domain of attraction of the Frechet distribution if and only if
\begin{align}
1-F_2(u)=u^{-\alpha}L_2(u),
\end{align}
where $\alpha>0$ and $L_2$ is a function of slow variation at infinity (cf. \cite{H70}).

If
$$
\varphi(x)=\varphi_1(x)\textbf{1}_{U(\zeta)}(x)+\varphi_2(x)\textbf{1}_{U(f^i(\zeta))}(x),
$$
and $X_0=\varphi$, then
$$
\mu(\varphi>u)=\mu(X_0>u)=u^pe^{-u^q}L_1(u)+u^{-\alpha}L_2(u)=u^{-\alpha}\left(L_2(u)+u^{\alpha+p}e^{-u^q}L_1(u)\right).
$$
As $\lim_{u\to\infty}u^{\alpha+p}e^{-u^q}{L_1(u)}=0$ we get that the function
$L(u)=L_2(u)+u^{\alpha+p}e^{-u^q}{L_1(u)}$ is a function of slow variation at infinity.
Hence, the distribution function of $\varphi(x)$ belongs to the domain of attraction of the Frechet distribution.
\end{proof}

\subsection{Gumbel versus Weibull}

\begin{proposition}
Let $\varphi:\mathcal{X}\to\mathbb{R}\cup\{\infty\}$ be an observable
which has the following form
$$
\varphi(x)=\varphi_1(x)\I_{U(\zeta)}(x)+\varphi_2(x)\I_{U(f^i(\zeta))}(x),
$$
where $\varphi_1$ is an observable defined in the neighborhood $U(\zeta)$ of $\zeta$,
$\varphi_2$ is an observable defined in the neighborhood  $U(f^i(\zeta))$ of
$f^i(\zeta)$ and they satisfy
$\max_x\varphi_1(x)= \varphi_1(\zeta)=D=\max_x\varphi_2(x)= \varphi_2(f^i(\zeta)) $.
Let $F_i$ be the distribution function of $\varphi_i$. If $F_1 $ and $F_2$ belong to the domain of attraction for maxima of the Gumbel and the Weibull distribution, respectively, then the distribution function of  $\varphi $, $F$, belongs to the domain of attraction of the Weibull distribution.
\end{proposition}
\begin{proof}

If
$X_0=\varphi$, then
$$
1-F(u)=\mu(\varphi>u)=\mu(X_0>u)=(1-F_1(u))+(1-F_2(u)).$$

So,
\begin{align}
\nonumber\lim_{s\to0}\frac{ 1-F(Dˆ-sy)}{1-F(D-s)}&=\lim_{s\to0}\frac{ 1-F_1(Dˆ-sy)+ 1-F_2(Dˆ-sy)}{1-F_1(D-s)+1-F_2(D-s)}\\
\label{quoc}&=\lim_{s\to0}\frac{  (1-F_2(Dˆ-sy))  \left(\frac{  1-F_1(Dˆ-sy)  }{  1-F_2(Dˆ-sy)  }     + 1 \right)   }{(1-F_2(Dˆ-s))  \left(\frac{  1-F_1(Dˆ-s)  }{  1-F_2(Dˆ-s)  }     + 1 \right)}.
\end{align}

As $F_2$ belongs to the domain of attraction of the Weibull distribution, then
\begin{align}
\label{Weibull_cond}\lim_{s\to0}\frac{  1-F_2(Dˆ-sy))  }{1-F_2(Dˆ-s)}=y^{\alpha},
\end{align}
for some $\alpha>0$, and
\begin{align}
\label{Weibull_tail} 1-F_2(Dˆ-s))=s^{\alpha}L(s),
\end{align}
where L is a function of slow variation at zero, i.e.,
for all $t>0$,
$$
\lim_{s\to0}\frac{L(ts)}{L(s)}=1.
$$

We note now that if $F_1$ belongs to the domain of attraction of the Gumbel distribution, with finite right endpoint of the support $D$, then, for all $\beta>0$,
\begin{align*}
\lim_{s\to0}\frac{1-F_1(D-s)}{s^{\beta}}=0.
\end{align*}
Thus, by (\ref{Weibull_tail}), we have that
\begin{align}
\label{GumbWeib}\lim_{s\to0}\frac{1-F_1(D-s)}{1-F_2(D-s)}=0.
\end{align}

Using now (\ref{Weibull_cond}) and (\ref{Weibull_tail}) in (\ref{quoc}), we obtain that
\begin{align}
\nonumber\lim_{s\to0}\frac{ 1-F(Dˆ-sy)}{1-F(D-s)}=y^{\alpha},
\end{align}
and so $F$ belongs to the domain of attraction of the Weibull distribution.

\end{proof}

\bibliographystyle{amsalpha}
\bibliography{multiplemaxima}

\end{document}